\documentclass[12pt,a4paper]{article}
\usepackage[utf8]{inputenc}
\usepackage{amsfonts,amssymb,amsmath,amsthm}
\usepackage[dvipsnames]{xcolor}
\usepackage{mathtools}
\usepackage{tikz}
\usepackage{tikz-3dplot}
\usetikzlibrary{matrix}
\usepackage{graphicx}
\usepackage{float}
\usepackage{blkarray}
\usepackage{url} % Include the url package for \url command
\usepackage{caption}  % Include the caption package for \captionof command
\usepackage{subcaption}
\usepackage{enumerate} % for custom enumerate
\usepackage{comment}
\usepackage{hyperref}
\newcommand{\vertiii}[1]{{\left\vert\kern-0.25ex\left\vert\kern-0.25ex\left\vert #1 
    \right\vert\kern-0.25ex\right\vert\kern-0.25ex\right\vert}}

\makeatletter
\def\nnfootnote{\xdef\@thefnmark{}\@footnotetext}
\makeatother

\usepackage{pgfplots}
\pgfplotsset{compat=newest}

\usepackage[inkscapearea=page]{svg}

\newcommand{\R}{\mathbb{R}}

\newtheoremstyle{nonitalic}% name
{3pt}% Space above
{3pt}% Space below
{}% Body font
{}% Indent amount
{\bfseries}% Theorem head font
{.}% Punctuation after theorem head
{.5em}% Space after theorem head
{}% Theorem head spec (can be left empty, meaning ‘normal’)

\theoremstyle{plain}
\newtheorem{theorem}{Theorem}[section] % Base counter

\theoremstyle{nonitalic}
\newtheorem{definition}[theorem]{Definition}
\newtheorem{example}[theorem]{Example}
\newtheorem{remark}[theorem]{Remark}

\theoremstyle{plain}
\newtheorem{corollary}[theorem]{Corollary}
\newtheorem{lemma}[theorem]{Lemma}
\newtheorem{proposition}[theorem]{Proposition}
\newtheorem{criterion}[theorem]{Criterion}

\theoremstyle{plain} % Or any other style you prefer
\newtheorem*{theorem*}{Theorem}
\newtheorem*{conjecture*}{Conjecture}

\usepackage[margin=2cm]{geometry} % alter page dimensions

\usepackage{pgfplots}
\pgfplotsset{compat=newest}

\usepackage[inkscapearea=page]{svg}

\title{Mathematical Foundations of Interlocking Assemblies}
\author{Tom Goertzen\footnote{RWTH Aachen University, Email: tom.goertzen@rwth-aachen.de}}
\date{}

\begin{document}

\maketitle

\begin{abstract}
    The study of interlocking assemblies is an emerging field with applications in various disciplines. However, to this day, the mathematical treatment of these assemblies has been sparse. In this work, we develop a comprehensive mathematical theory for interlocking assemblies, providing a precise definition and a method for proving the interlocking property based on infinitesimal motions. We consider assemblies with crystallographic symmetries and verify interlocking properties for such assemblies. Our analysis includes the development of an infinite polytope with crystallographic symmetries to ensure that the interlocking property holds. For a certain block, called the RhomBlock, that can be assembled in numerous ways, characterised by the combinatorial theory of lozenges, we rigorously prove the interlocking property. By conclusively showing that any assembly of the RhomBlock is interlocking, we provide a robust framework for further exploration and application of interlocking assemblies.
\end{abstract}

\section{Introduction}

An interlocking assembly consists of rigid bodies, called \emph{blocks}, which interlock purely based on their geometric properties and an additional constraint that fixes a certain set of blocks from moving, the \emph{frame}. This means that no finite subset of blocks not contained in the frame can be moved without causing penetrations with other blocks.

Interlocking assemblies, also known as (topological) interlocking assemblies in engineering applications, show promising features when applied in various areas, see \cite{EstrinDesignReview}. For instance, besides applications in architecture, see \cite{fallacara_topological_2019,lecci_design_2021}, the concept of interlocking assemblies can also be applied to extraterrestrial constructions \cite{dyskin_principle_2005} or to interlocking puzzles \cite{InterlockingPuzzles}.

However, the mathematical theory behind these assemblies remains rather unexplored, and verifying the interlocking property leads to a high-dimensional problem as we have to consider motions for each block simultaneously.

There are several construction methods yielding candidates for interlocking assemblies, see \cite{EstDysArcPasBelKanPogodaevConvex,akleman_generalized_2020,GoertzenFIB,goertzen2024constructing}. A common feature many of these constructions have is the appearance of wallpaper symmetries, which correspond to planar crystallographic groups. In \cite{GoertzenFIB}, a definition for (topological) interlocking assemblies is given in terms of continuous motions. In \cite{wang_design_2019,wang_computational_2021}, an approach for verifying the interlocking property based on infinitesimal motions is given, leading to solving a linear optimisation problem.

In this work, we connect both of these terminologies and show that we can indeed translate the definition in \cite{GoertzenFIB} into the infinitesimal criterion given in \cite{wang_computational_2021}, by considering continuous motions which are differentiable in their starting configuration.

This tool can be used to verify interlocking properties of assemblies constructed with the methods in \cite{goertzen2024constructing}. There, the idea is to deform two fundamental domains $F,F'$ of a planar crystallographic group $G$ acting on the Euclidean plane $\mathbb{R}^2$ continuously into each other to obtain an interlocking block $X$, which can be assembled by extending the action of $G$ onto $\mathbb{R}^3$. To be more specific, we start with a fundamental domain $F$ in the shape of a polyhedron and deform the edges of $F$ in the spirit of M.C. Escher in a path-based framework, to obtain a new domain $F'$. If the deformation comes from piecewise-linear paths, we can place $F$ and $F'$ into two parallel planes in $\mathbb{R}^3$ and interpolate between them using a triangulation. In Figure \ref{fig:p6Example}, we see an example of this construction. 

\begin{figure}[H]
\centering
\begin{minipage}{.2\textwidth}
  \centering
  \includegraphics[height=5cm]{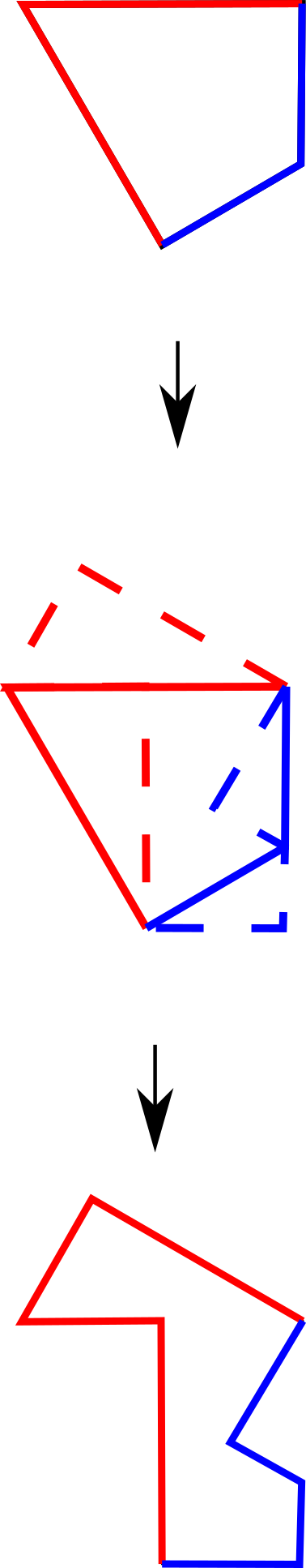}
  \subcaption{}
    \label{fig:p6EscherTrick}
\end{minipage}%
\begin{minipage}{.25\textwidth}
  \centering
  \includegraphics[height=5cm]{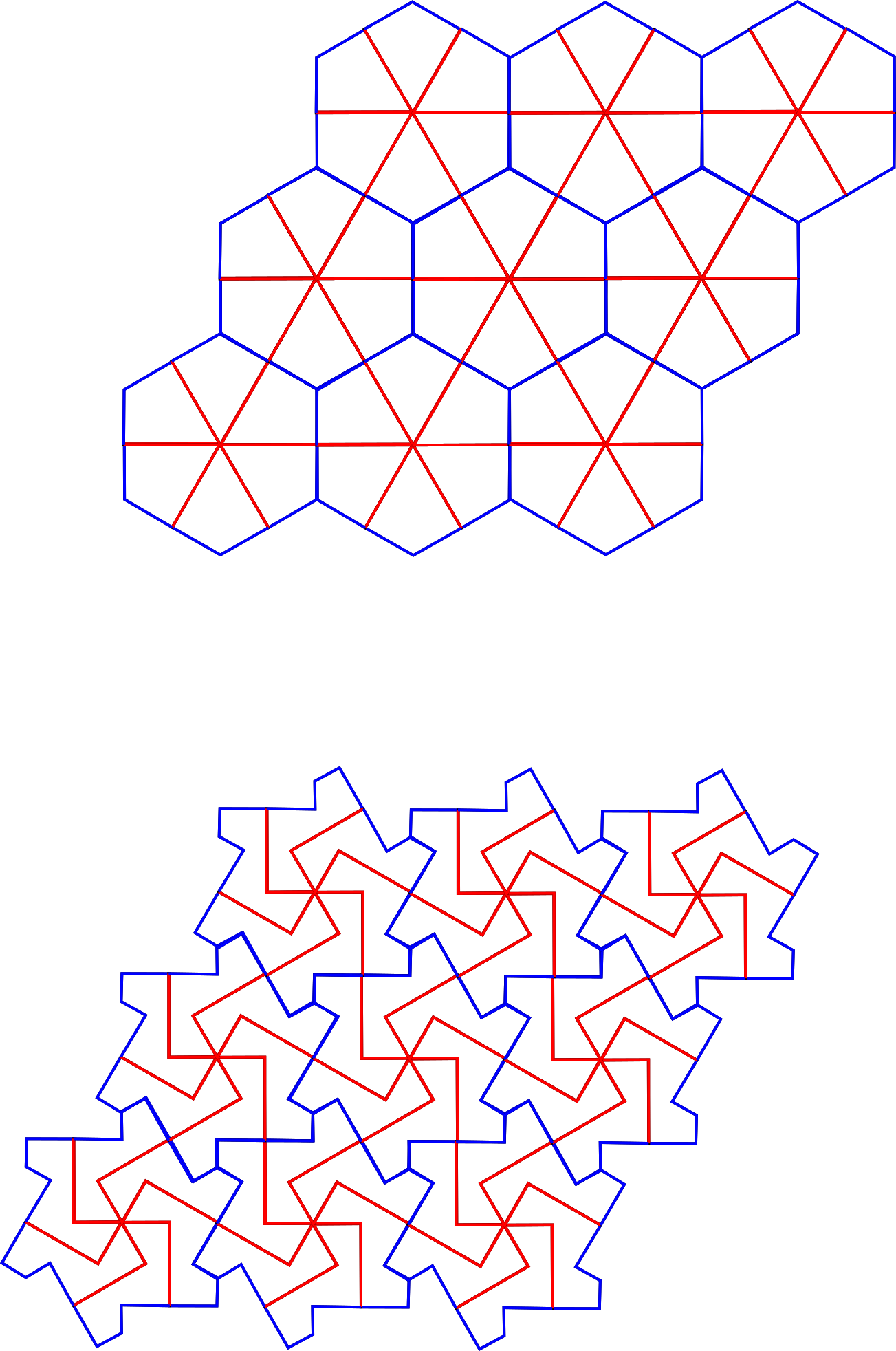}
  \subcaption{}
    \label{fig:P6Domains}
\end{minipage}
\begin{minipage}{.2\textwidth}
  \centering
  \includegraphics[height=5cm]{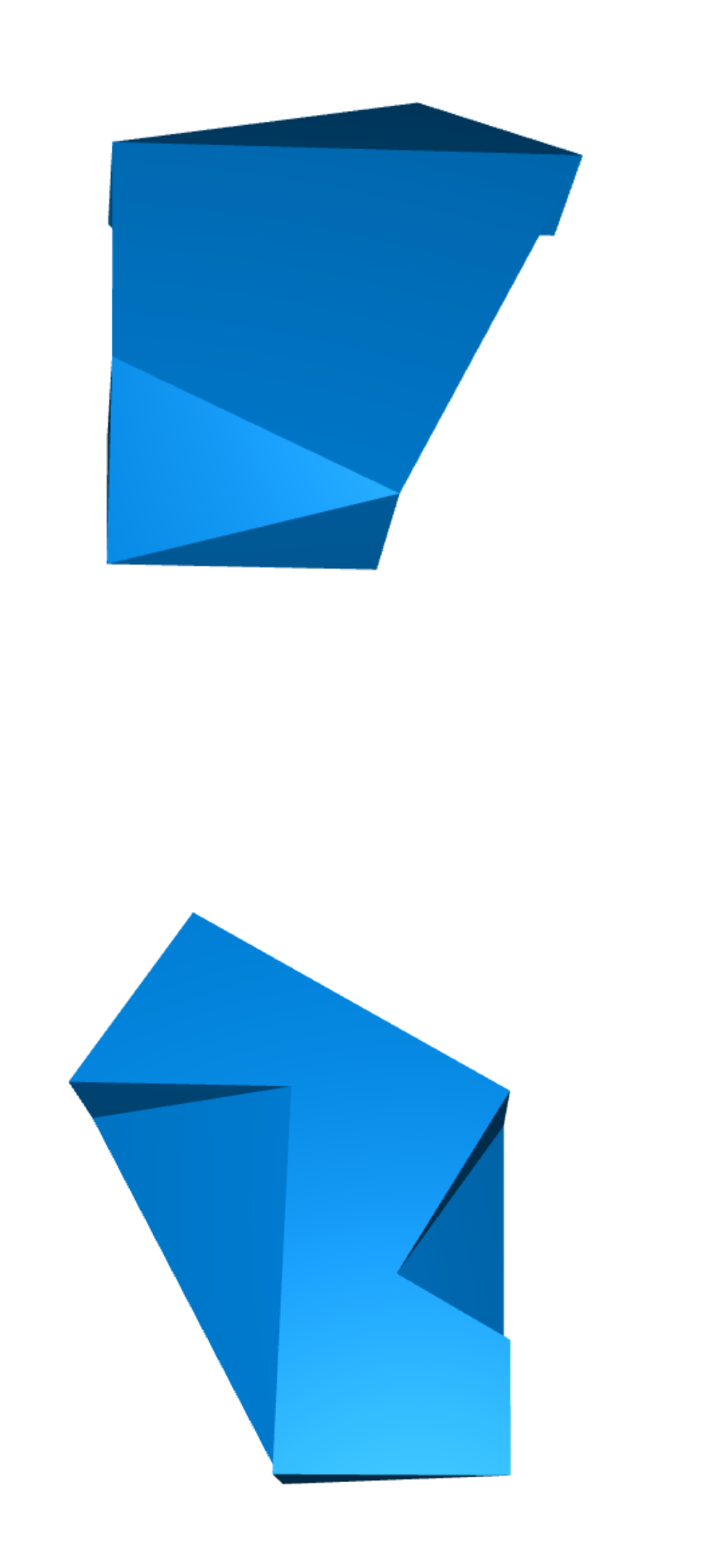}
  \subcaption{}
    \label{fig:P6Block}
\end{minipage}
\begin{minipage}{.3\textwidth}
  \centering
  \includegraphics[height=5cm]{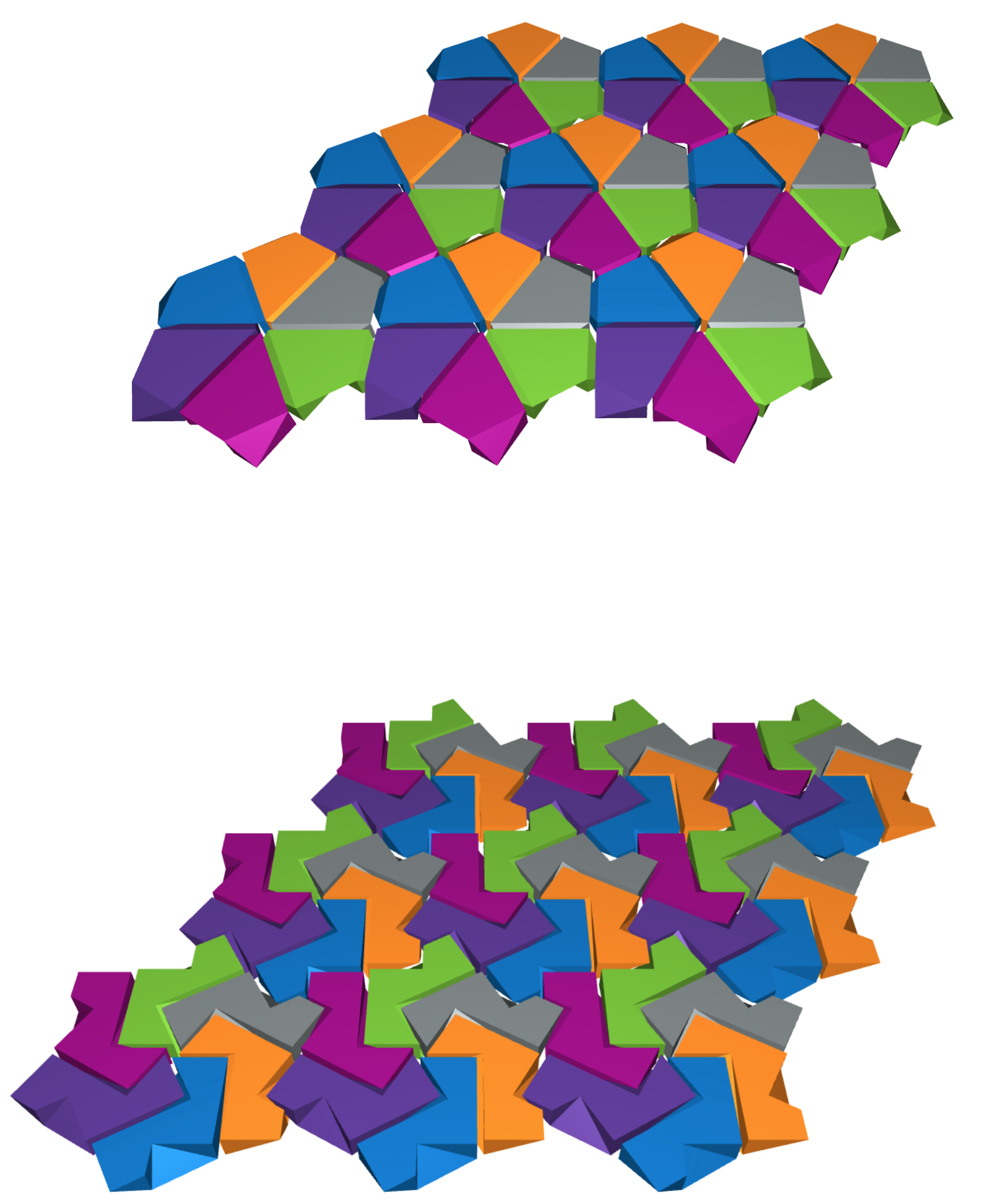}
  \subcaption{}
    \label{fig:P6Assembly}
\end{minipage}
\caption{Demonstrating the steps for generating interlocking blocks by exploiting a planar crystallographic group of type p6 (using the international notation \cite{ITA2002}), see \cite{goertzen2024constructing}: (a) First, we apply the \emph{Escher Trick}, i.e.\ obtain a new fundamental domain from a given one. (b) Both domains yield tessellations of the plane. (c) Next, a block is obtained by interpolating between the two domains. (d) We obtain an assembly with blocks coloured according to their arrangement.}
\label{fig:p6Example}
\end{figure}

The assemblies constructed with the method given in \cite{goertzen2024constructing} consist of infinitely many blocks, with each block corresponding to a group element of the underlying planar crystallographic group. 

In Section \ref{sec:proof}, we show that deforming all edges of a fundamental domain in the construction mentioned above, leads to an assembly of blocks, where we cannot move blocks using so-called ``sliding-motions'' and verify a translational interlocking property in the special case that the underlying group $G$ is generated by two translations and isomorphic to $\mathbb{Z}^2$.

In \cite{goertzen2024constructing}, a family of tiles is constructed called \emph{VersaTiles} that can be used to construct blocks that can be assembled in many ways, characterised by generalised Truchet tiles. In one boundary case, these generalised Truchet tiles are exactly given by \emph{lozenges}, i.e.\ a rhombus obtained by two equilateral triangles. Furthermore, a block is introduced, called \emph{RhomBlock}, whose assemblies correspond directly to tessellations with lozenges, see Figure \ref{fig:RhomBlock}.

\begin{figure}[H]
\centering
\begin{minipage}{.4\textwidth}
\begin{figure}[H]
    \centering
   \resizebox{!}{2.5cm}{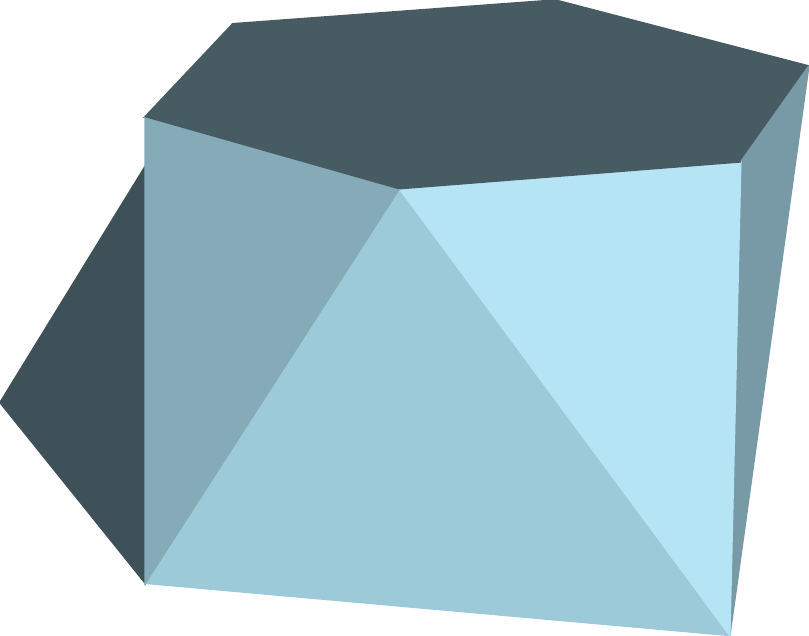}
\end{figure}    
\end{minipage}
\begin{minipage}{.4\textwidth}
\begin{figure}[H]
    \centering

\resizebox{!}{2.5cm}{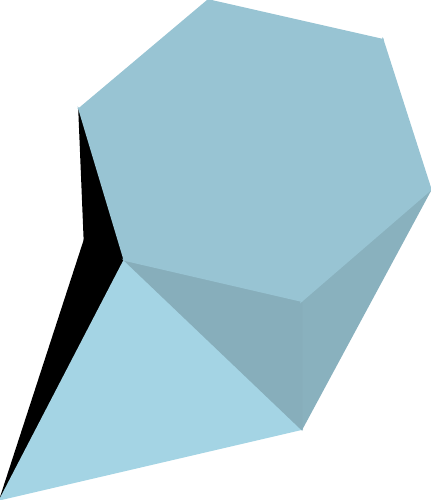}
\end{figure}    
\end{minipage}
\caption{Two views of the RhomBlock constructed in \cite{goertzen2024constructing}.}
\label{fig:RhomBlock}
\end{figure}

In Section \ref{sec:proof}, we show that any infinite assembly of RhomBlocks coming from a tessellation with lozenges indeed gives rise to an interlocking assembly with infinite blocks. We can take finite subsets of these assemblies to construct finite interlocking assemblies. In Figure \ref{fig:RhomBlockAssembly}, we show an example of an interlocking assembly with $10 \times 10$ (in the hexagonal grid) RhomBlocks.

\begin{figure}[H]
\centering
\begin{subfigure}{.33\textwidth}
  \centering
  \includegraphics[height=2.5cm]{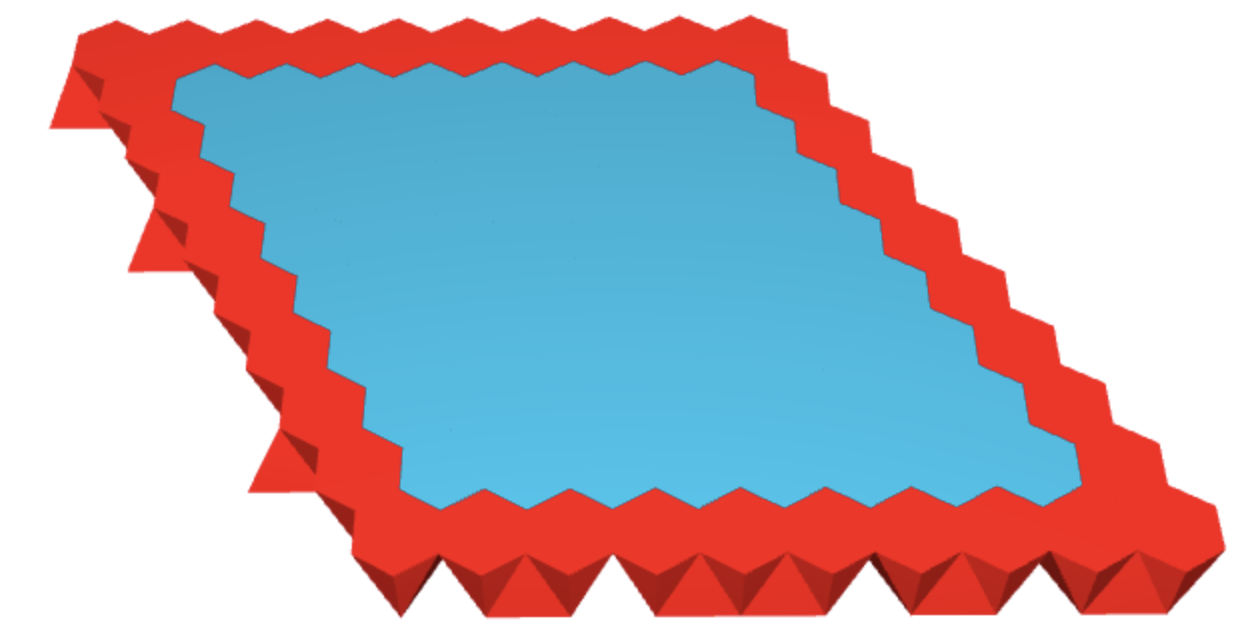}
  \subcaption{}
\end{subfigure}%
\begin{subfigure}{.33\textwidth}
  \centering
  \includegraphics[height=2.5cm]{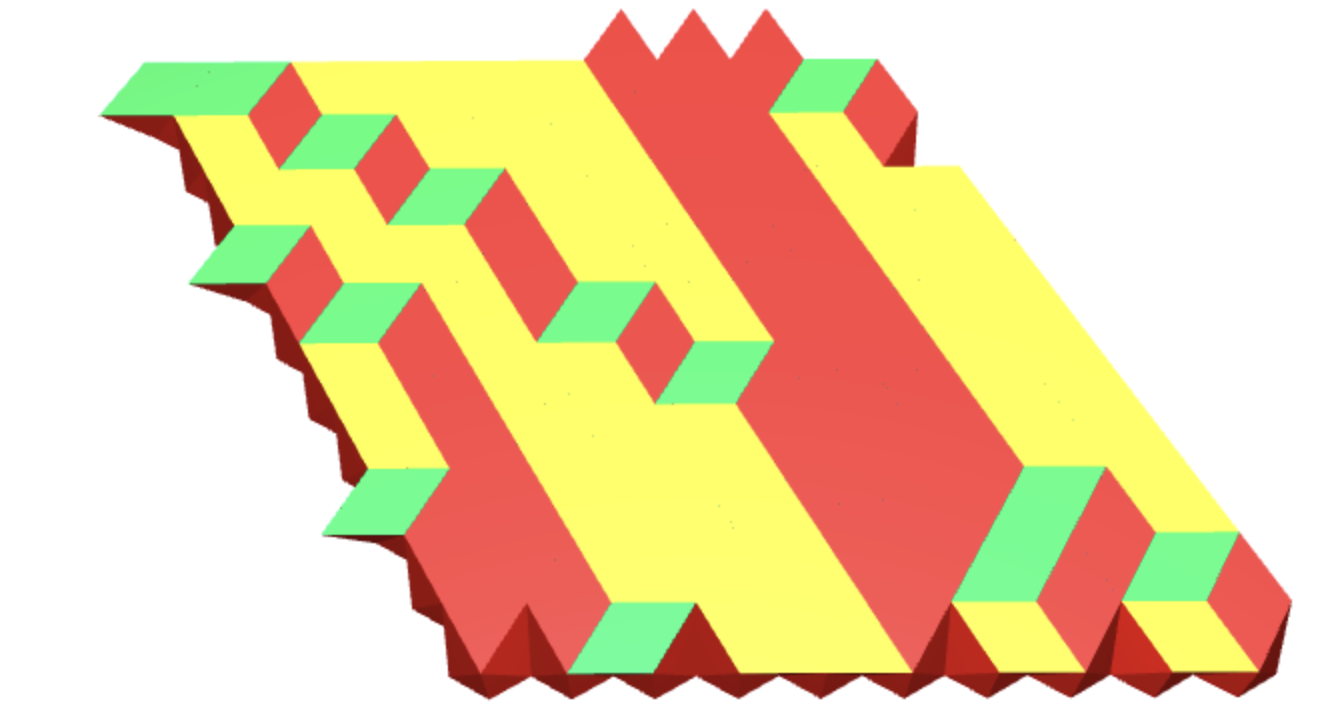}
  \subcaption{}
\end{subfigure}
\begin{subfigure}{.33\textwidth}
  \centering
  \includegraphics[height=2.5cm]{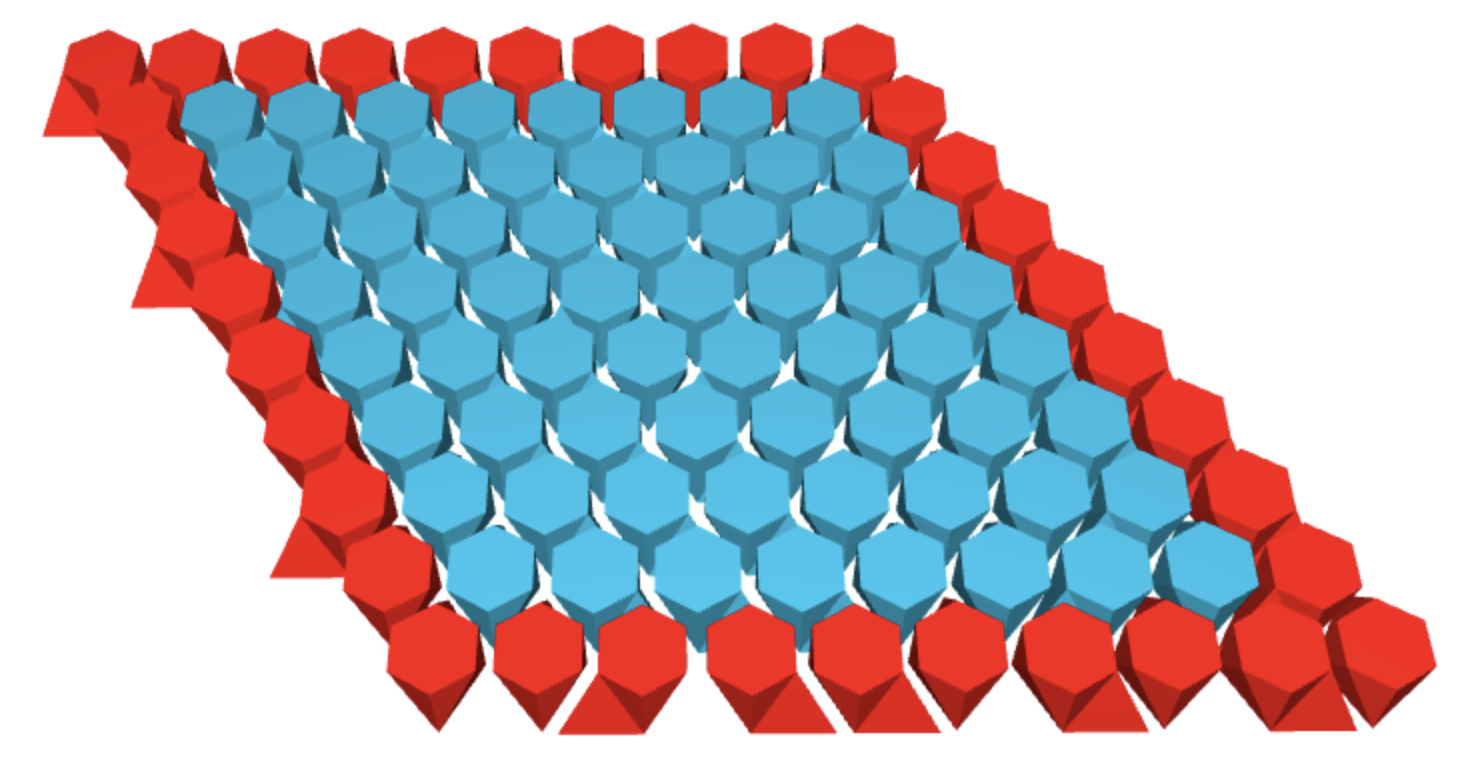}
  \subcaption{}
\end{subfigure}
\caption{(a) Interlocking assembly with RhomBlocks. (b) Corresponding lozenge tiling with lozenges coloured according to their orientation. (c) Exploded view of the assembly.}
\label{fig:RhomBlockAssembly}
\end{figure}

In Section \ref{sec:definition}, we provide a mathematical description of interlocking assemblies based on the work presented in \cite{GoertzenFIB}. In Section \ref{sec:infinitesimal}, we establish a connection between the mathematical definition of interlocking assemblies, first presented in \cite{GoertzenFIB}, and an infinitesimal definition of such assemblies formulated in \cite{wang_computational_2021}. In Section \ref{sec:proof}, we prove that with additional restrictions on the construction of assemblies with planar crystallographic symmetries, we obtain the interlocking assemblies.

\section{Definition of Interlocking Assemblies} \label{sec:definition}

In this section, we define the notion of interlocking assemblies. Intuitively speaking, an interlocking assembly is an assembly of blocks such that a subset of blocks, the \emph{frame}, is fixed and no block outside the frame can be removed from the assembly by continuous motions without intersecting other blocks. Before giving the formal definition of interlocking assemblies, we first have to define what an assembly of blocks is and what a motion is. We present several examples of assemblies which either possess the interlocking property or do not. In the last part of this section, we demonstrate how to prove that an assembly possesses the interlocking property based on infinitesimal motions. Parts of this section are based on the work presented in \cite{GoertzenFIB}.

\subsection{Rigid and Continuous Motions}

Given a subset $X \subset \mathbb{R}^3$, we want to describe motions of $X$ in three-dimensional Euclidean space. It is well-known that a rigid (non-continuous) motion can be obtained by composing a rotation and a translation given by a three-dimensional vector. Altogether, these motions form the so-called \emph{special Euclidean group}. Note that, a rotation in three-dimension space around the origin corresponds to a matrix $R\in \R^{3\times 3}$ such that $R$ is orthogonal and has determinant equal to $1$, i.e.\ $R\cdot R^\intercal = \mathbb{I}$ and $\det R=1$. The set of all rotational matrices of this form is denoted by $\mathrm{SO}(3)$.

\begin{definition}
    The \emph{special Euclidean group} also known as the set of rigid motions, denoted by $\mathrm{SE}(3)$, consists of  elements which can be represented as pairs $(R,v)$, where $R\in \mathrm{SO}(3)$ is a rotation matrix and $v$ a translation vector. The group $\mathrm{SE}(3)$ is isomorphic to the semidirect product $\mathrm{SO}(3)\ltimes \R^3$ with multiplication given by $(R,v)\cdot (R',v')=(R\cdot R',R\cdot v' + v)$.
\end{definition}

 The \emph{special Euclidean group} acts on $\R^3$ as summarised in the next remark.

\begin{remark}\label{rem:affine_action}
   The \textit{special Euclidean group} acts naturally on $\R^3$ via the following group action
   $$\mathrm{SE}(3) \times \R^3 \to \R^3, \left( (R,v),x \right)\mapsto R\cdot x + v,$$  where $R\cdot x$ is the vector $x$ rotated by the rotation matrix $R$. The following matrix representation of  $\mathrm{SE}(3)$ for a fixed basis of $\mathbb{R}^3$ can be used to encode the above action
\begin{equation}
    \rho:\mathrm{SE}(3) \to \mathrm{GL}(4), 
    (R,v) \mapsto \begin{pmatrix} R & v \\ 0 & 1   \end{pmatrix}.
\end{equation} 
 For this, we identify $\mathbb{R}^3$ with the set $\mathrm{Aff}\left(\R^3 \right) =\left\{ \begin{pmatrix} x \\ 1 \end{pmatrix} \mid x \in \mathbb{R}^3 \right\}\subset \mathbb{R}^4$, also known as the \emph{affine space}, and $\mathrm{SE}(3)$ then acts on $\R^3$ as follows
$$\mathrm{SE}(3)\times \R^3 \to \R^3, \left( (R,v) ,  \begin{pmatrix} x \\ 1 \end{pmatrix} \right) \mapsto \rho\left((R,v)\right) \cdot \begin{pmatrix} x \\ 1 \end{pmatrix}= \begin{pmatrix} R & v \\ 0 & 1   \end{pmatrix}\cdot \begin{pmatrix} x \\ 1 \end{pmatrix}=\begin{pmatrix} R\cdot x + v \\ 1 \end{pmatrix}.$$From now on, we identify $\mathrm{SE}(3)$ with its image of the above-mentioned matrix representation into $\mathrm{GL}(4)$. By $\mathbb{I}\in \mathrm{GL}(4)$, we denote the identity matrix and call it the \textit{trivial motion}. If the context allows, we refer to certain elements in $\mathrm{SE}(3)$ just by their corresponding rotation matrices or translation vectors.
The set of rigid motions $\text{SE}(3)$ inherits a topology as an isomorphic image of a subset of $\mathbb{R}^{4\times 4}$ equipped with the operator norm $\vertiii{  \cdot} _2 $ which, for a given matrix $A\in \mathbb{R}^{4\times 4}$, is defined via
$$\vertiii{  A} _2 = \max _{x\in \mathbb{R}^4,\lVert x \rVert _2= 1} \lVert A\cdot x \rVert _2 ,$$ where $A\cdot x \in \mathbb{R}^4$ is the matrix-vector product of $A$ and $x$, and $$\lVert x \rVert _2 = \sqrt{x_1^2+x_2^2+x_3 ^2+x_4^2}$$ is the Euclidean norm on $\mathbb{R}^4$.
\end{remark}

As shown in the previous remark, the elements of the \textit{special Euclidean group} act on $\R^3$ as rigid motions. In order to define immovability of blocks in an assembly, we continue with the more applicable definition of continuous motions.

\begin{definition}
   A \emph{continuous motion} is a map $\gamma \colon [0,1]\to \mathrm{SE}(3)\subset \mathrm{GL}(4)$ such that $\gamma$ is continuous (using the topology given in Remark \ref{rem:affine_action}) and $\gamma(0)=\mathbb{I}$ is the identity matrix in $\mathrm{SE}(3)$. Furthermore, we say $\gamma$ is \emph{admissible} if $\gamma$ is a continuous motion and differentiable in $0$. We say that $\gamma$ is trivial if $\gamma(t)=\mathbb{I}$ for all $t\in [0,1]$ and we write $\gamma \equiv \mathbb{I}$. For admissible motions, we further enforce that for non-trivial maps $\gamma$  the derivate at zero is non-zero, i.e.\ $\Dot{\gamma}(0)\neq 0$.
 \end{definition}

The assumption that the first derivative of a non-trivial admissible motion does not vanish is needed to establish a connection between the usual definition of interlocking assemblies and the infinitesimal version, see proof of Proposition \ref{proposition:infinitesimal}.

\begin{remark}\label{rem:lie_algebra}
    The group $\mathrm{SE}(3)$ is a Lie group. An admissible motion $\gamma$ can be differentiated in $0$ to obtain an element in the corresponding Lie algebra  $\mathfrak{se}(3)$. This is due to the fact, that we can extend the map $\gamma$ to a differentiable map $\Tilde{\gamma}:[-1,1]\to \text{SE}(3)$ by 
    \[
    \tilde{\gamma}\left(t\right)=\begin{cases}
    \gamma\left(t\right), & t\geq0,\\
    \gamma\left(-t\right)^{-1}, & t<0
    \end{cases}
    \]
    and the definition of the Lie algebra as the tangent space at the identity element of the underlying Lie group, see \cite[Chapter~1.4]{DifferentialGeometryKobayashi}.
    The Lie algebra $\mathfrak{se}(3)$ is a $6-$dimensional vector space with elements of the form $(\omega,t)=(\omega_1,\omega_2,\omega_3,t_1,t_2,t_3)$ which can be embedded into the $\R^{4\times 4}$ matrix-space as follows
\[
\begin{pmatrix}0 & -\omega_{3} & \omega_{2} & t_{1}\\
\omega_{3} & 0 & -\omega_{1} & t_{2}\\
-\omega_{2} & \omega_{1} & 0 & t_{3}\\
0 & 0 & 0 & 0
\end{pmatrix}
\]
and thus we can define a multiplication of elements $\left(\omega,t\right)\in\mathfrak{se}(3)$ with elements in $p\in \R^3$ via
\[
\left(\omega,t\right).p=\begin{pmatrix}0 & -\omega_{3} & \omega_{2} & t_{1}\\
\omega_{3} & 0 & -\omega_{1} & t_{2}\\
-\omega_{2} & \omega_{1} & 0 & t_{3}\\
0 & 0 & 0 & 0
\end{pmatrix} \cdot \begin{pmatrix}p \\ 1\end{pmatrix}=\begin{pmatrix} \omega\times p+t \\ 1\end{pmatrix},
\]
where $\times:\mathbb{R}^3\times \mathbb{R}^3 \to \mathbb{R}^3$ denotes the cross product given by

\[
\omega\times p=\begin{pmatrix}\omega_{2}p_{3}-\omega_{3}p_{2}\\
\omega_{3}p_{1}-\omega_{1}p_{3}\\
\omega_{1}p_{2}-\omega_{2}p_{1}
\end{pmatrix}.
\]
For more on the correspondence of $\mathrm{SE}(3)$ and $\mathfrak{se}(3)$, we refer to \cite[Chapter~6]{HarmonicAnalysisApplied}.

\end{remark}

\subsection{Assemblies of Blocks}

In order to define interlocking assemblies, we first need to define the notion of an assembly of blocks.

\begin{definition}\label{def:block}
    Let $\emptyset\neq X\subset \R^n$ be a connected, compact set (in the standard Euclidean topology) with $\overline{\mathring{X}}=X$, i.e.\ $X$ equals the closure of its interior. We call $X$ a \emph{block} with boundary denoted by $\partial X$. 
\end{definition}

The constraints applied to a block $X$ are inspired by practical applications and the geometric shapes of objects in three-dimensional space. Additionally, the use of the condition $\overline{\mathring{X}}=X$ serves to exclude any type of degenerations. Frequently, extra restrictions are imposed on $X$. For example, when $n=3$, we focus on blocks $X$ that have a polyhedral boundary.

Next, we define assemblies of blocks.

\begin{definition}
    An \emph{assembly} is a family of blocks $(X_i)_{i\in I}$ for a non-empty countable index set $I$ such that $X_i \cap X_j = \partial X_i \cap \partial X_j$ for all $i,j\in I$ with $i\neq j$.
\end{definition}
 This condition enforces that two distinct blocks of an assembly only touch at their boundary and do not ``penetrate'' each other, i.e.\ their interiors do not intersect. Furthermore, we allow infinite assemblies of blocks, which is compatible with the assemblies constructed in \cite{goertzen2024constructing} that carry a doubly-periodic symmetry.

In the following, we provide several intuitive examples of interlocking and non-interlocking assemblies before giving the formal definition of an interlocking assembly in the following section. We start with two ways of assembling cubes.

Figure \ref{fig:CubeAssembly} displays a canonical way of assembling cubes in a doubly-periodic way, i.e.\ we translate a given cube in two directions using two vectors $x,y\in \R^3$. When assembling cubes, as shown in Figure \ref{fig:CubeAssembly}, it is always possible to move cubes by shifting them upwards, even when neighbouring cubes are constrained from moving. In Figure \ref{fig:CubeInterlocking}, we see an alternative way of assembling cubes in a doubly-periodic fashion. We can place the assembly between two parallel planes such that the midsection of each cube, i.e.\ its intersection with the plane going through the middle of the assembly, is given by a hexagon. This assembly can be generated with the method presented in \cite{EstDysArcPasBelKanPogodaevConvex}, which is based on the well-known statement that any convex body can be constructed by a finite intersection of half-spaces, see \cite{GrunbaumPolytopes}. In Example \ref{example:cube_interlocking}, we show the assembly in Figure \ref{fig:CubeInterlocking} indeed gives rise to an interlocking assembly.

\begin{figure}[H]
\centering
\begin{minipage}{.4\textwidth}
  \centering
  \includegraphics[height=3cm]{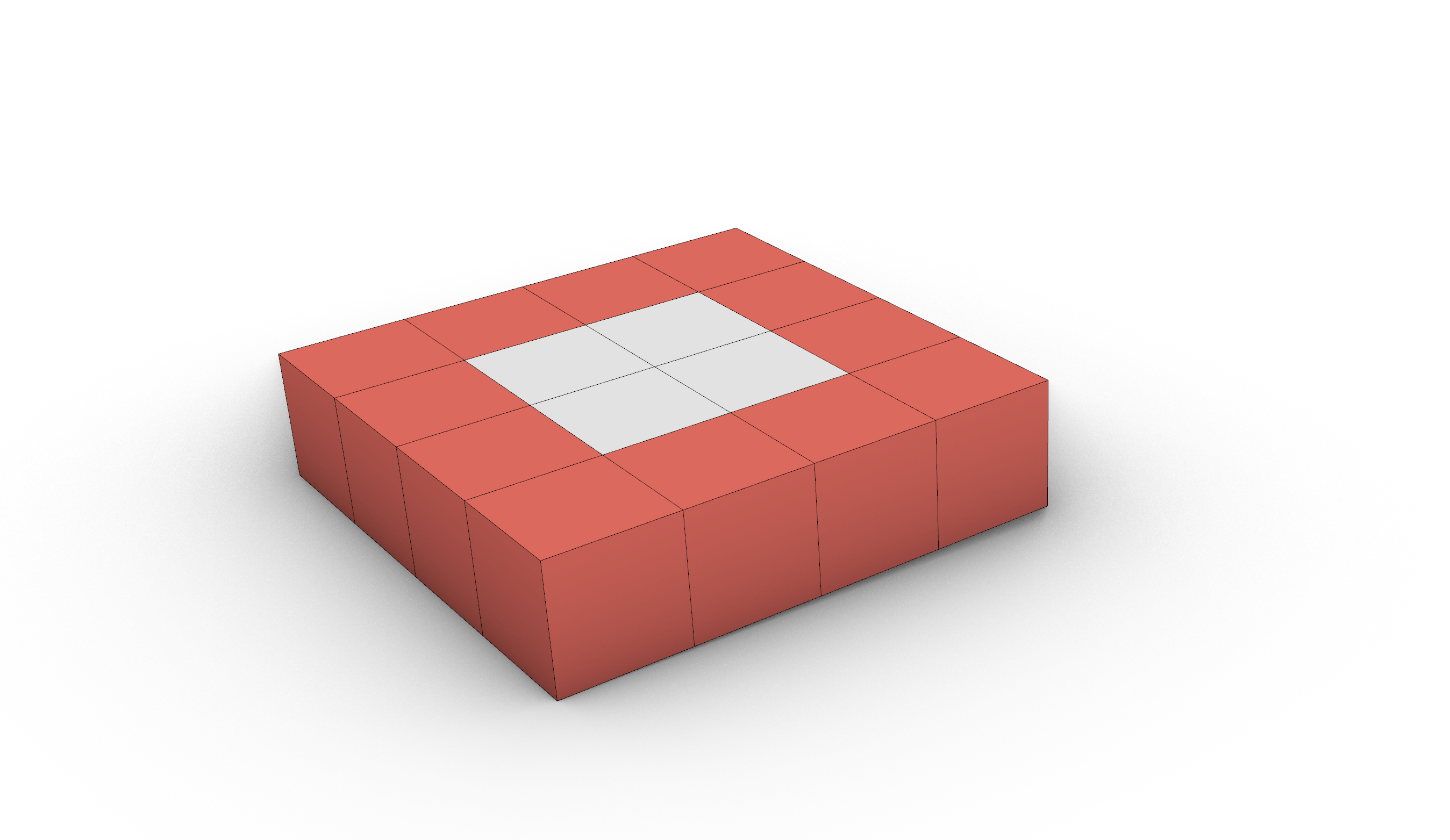}
  \subcaption{}
  %\captionof{figure}{Assembly of cubes}
    \label{fig:CubeAssembly}
\end{minipage}%
\begin{minipage}{.4\textwidth}
  \centering
  \includegraphics[height=3cm]{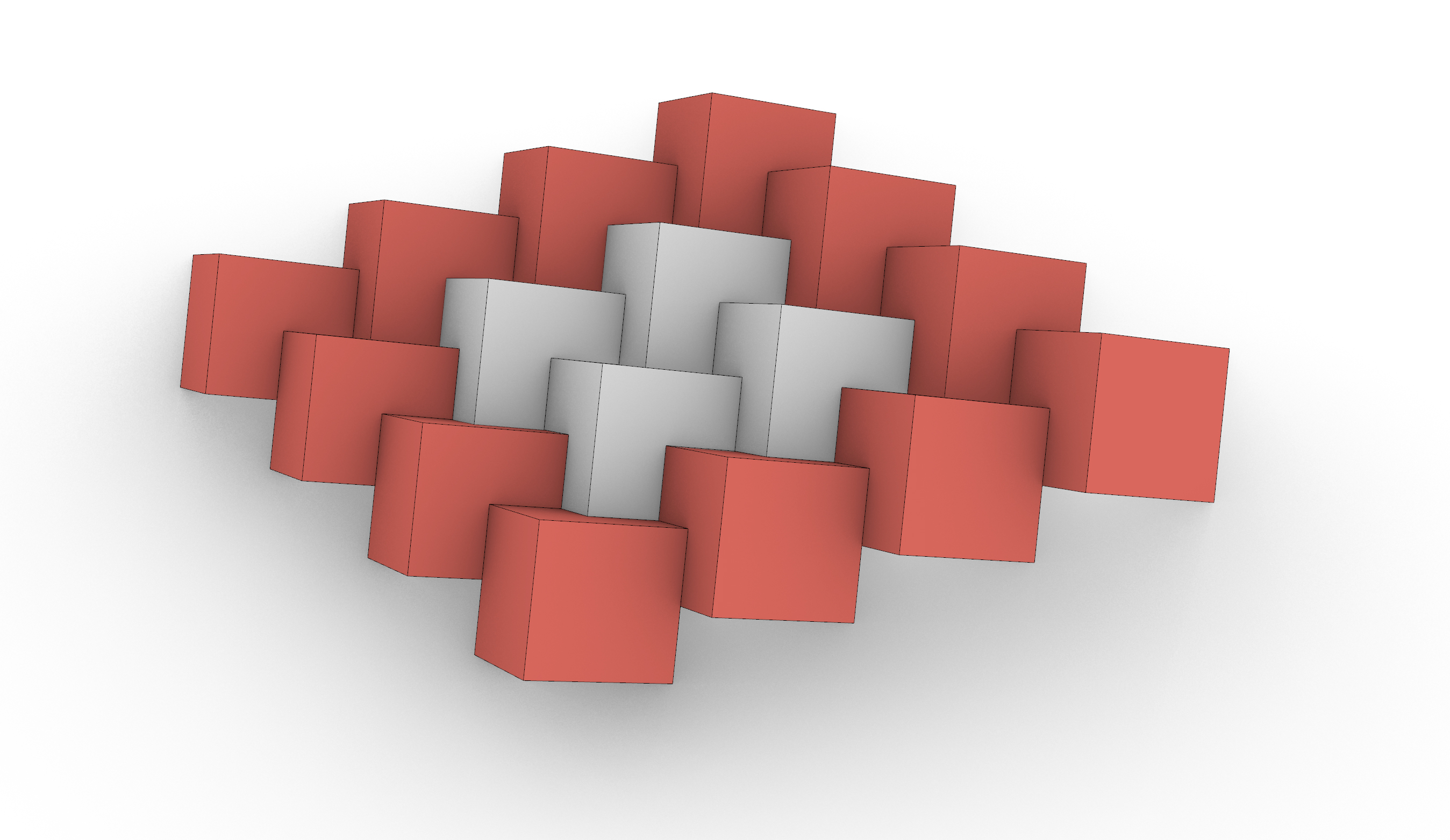}
  \subcaption{}
    \label{fig:CubeInterlocking}
\end{minipage}
\caption{Two ways of assembling cubes in a doubly-periodic fashion: (a) a simple cube assembly where grey cubes can be moved even when fixing the red cubes from moving. (b) An interlocking assembly of cubes which can be generated by the methods presented in \cite{EstDysArcPasBelKanPogodaevConvex}.}
\label{fig:cube-assemblies}
\end{figure}

 An assembly of regular tetrahedra as shown in Figure \ref{fig:p4Tetrahedra} (see \cite{a_v_dyskin_topological_2003,glickman_g-block_1984}) gives another example of an interlocking assembly by choosing a frame consisting of the red tetrahedra to be fixed. Then it follows that the grey blocks cannot be removed from the assembly by continuous motions without causing penetrations with other blocks. Figure \ref{fig:p3Tetrahedra} and \ref{fig:p6Tetrahedra} show similar assemblies with non-regular tetrahedra, i.e.\ tetrahedra whose faces are not equilateral triangles. A frame is not shown in Figure \ref{fig:p3Tetrahedra} and \ref{fig:p6Tetrahedra}, and a possible frame could consist of the outer tetrahedra. All the assemblies shown in Figures \ref{fig:CubeInterlocking}-\ref{fig:p6Tetrahedra} are based on regular tessellations of the plane with wallpaper symmetries and can be obtained by the intersections of half-spaces as described in \cite{EstDysArcPasBelKanPogodaevConvex}.

\begin{figure}[H]
\centering
\begin{subfigure}{.33\textwidth}
  \centering
  \includegraphics[height=3cm]{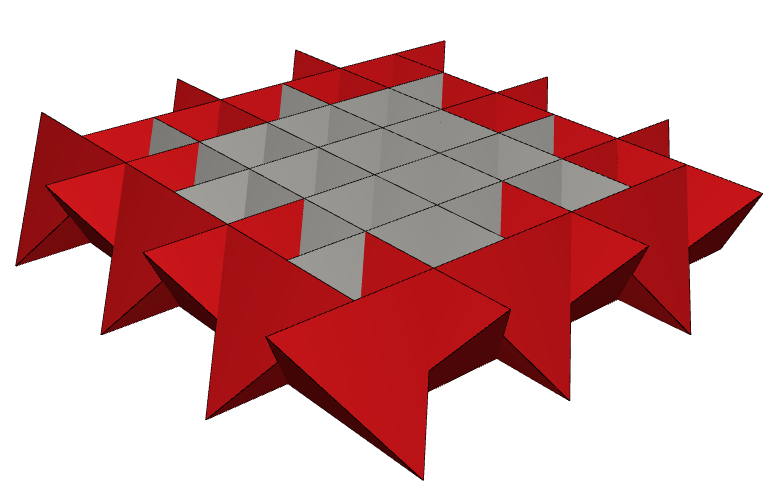}
  \subcaption{}
  \label{fig:p4Tetrahedra}
\end{subfigure}%
\begin{subfigure}{.33\textwidth}
  \centering
  \includegraphics[height=3cm]{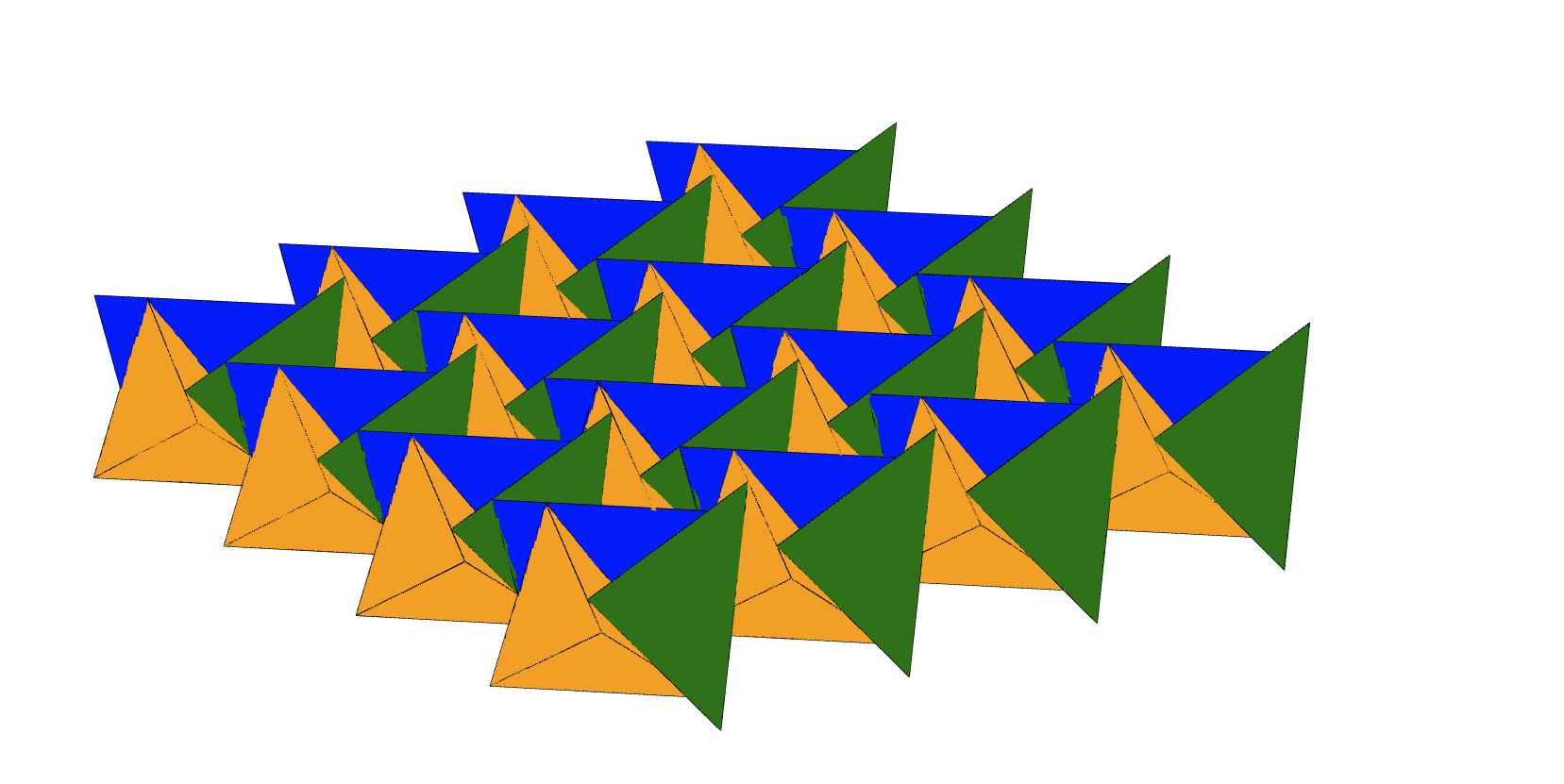}
  \subcaption{}
  \label{fig:p3Tetrahedra}
\end{subfigure}
\begin{subfigure}{.33\textwidth}
  \centering
  \includegraphics[height=3cm]{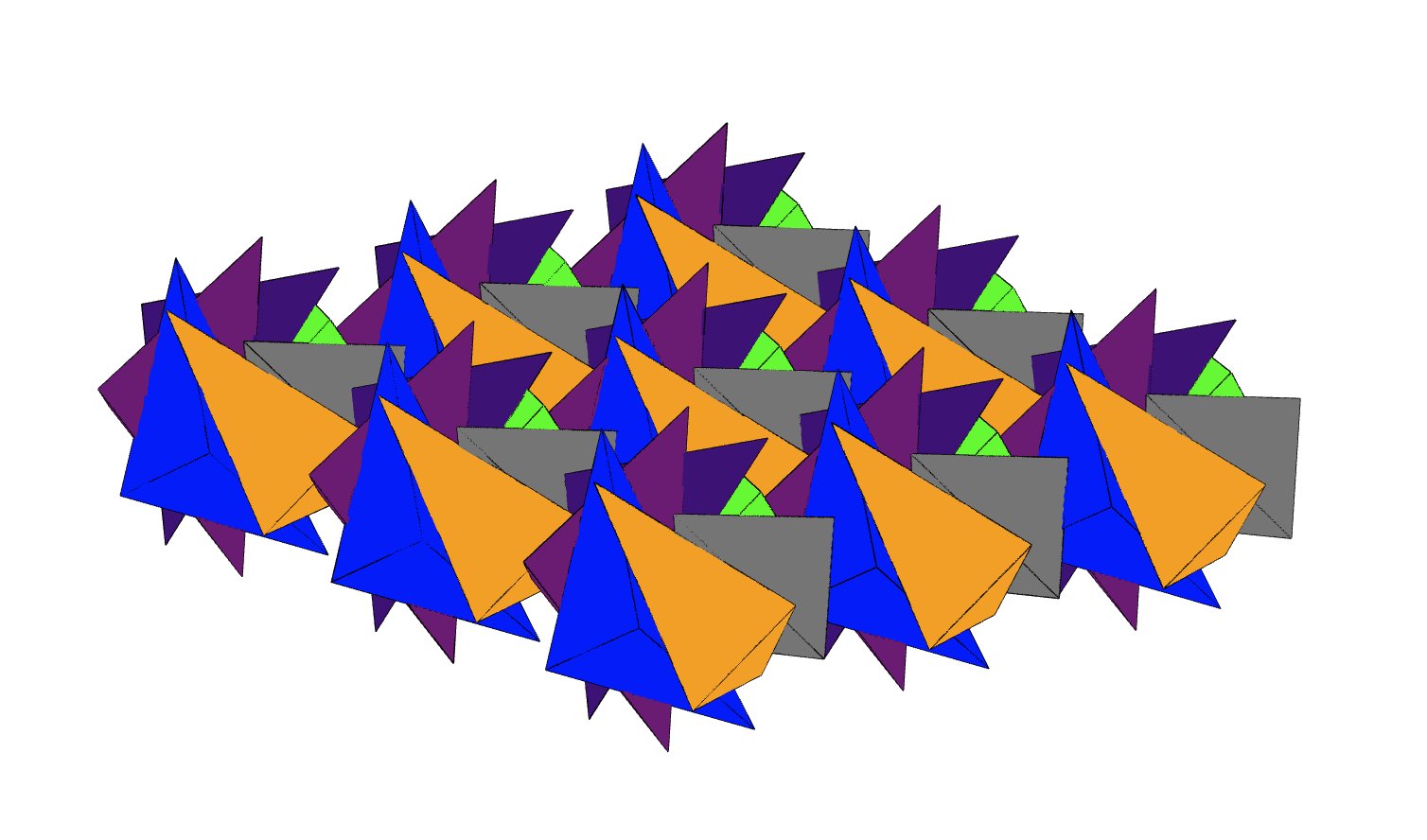}
  \subcaption{}
  \label{fig:p6Tetrahedra}
\end{subfigure}
\caption{Collection of assemblies with tetrahedra: (a) regular tetrahedra interlocking with p4-symmetry frame in red, see \cite{glickman_g-block_1984,dyskin_new_2001}, (b) non-regular tetrahedra assembly with p3-symmetry, see \cite{piekarski_floor_2020} for a similar block, (c) non-regular tetrahedra assembly with p6-symmetry.}
\label{fig:tetrahedra}
\end{figure}

Before we define interlocking assemblies formally, we provide another example of an assembly of modified cubes which does yield an interlocking assembly, even though a single block cannot be removed from the assembly while fixing all other blocks and especially its neighbouring blocks. This block is obtained by modifying a cube in such a way that on one side a pyramid is added, which is removed from a different side of the cube, see Figure \ref{fig:modified_cube}.
\begin{figure}[H]
    \centering
    \begin{minipage}{.4\textwidth}
        \centering
        \includegraphics[height=4cm]{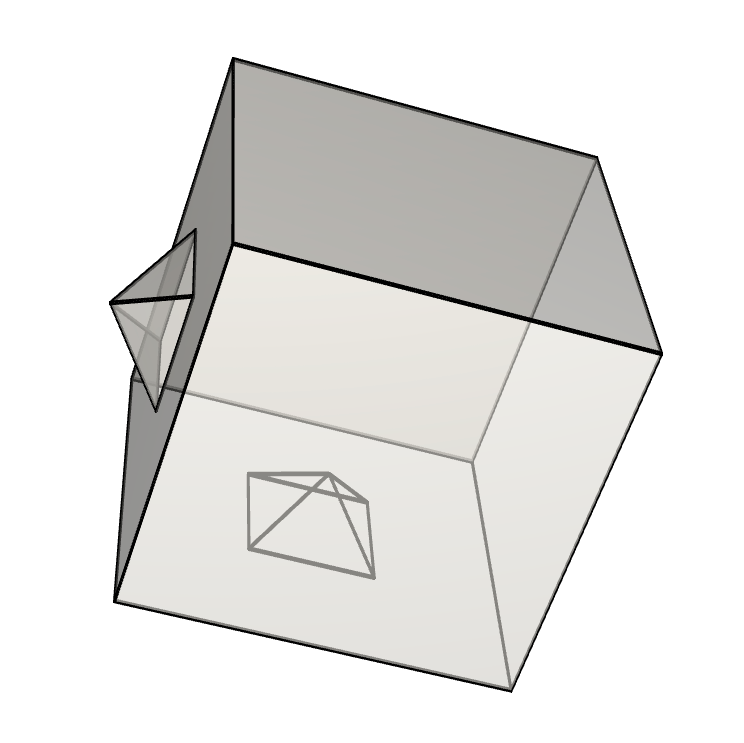}
        \subcaption{}
        \label{fig:modified_cube}
    \end{minipage}
    \begin{minipage}{.4\textwidth}
        \centering
        \includegraphics[height=4cm]{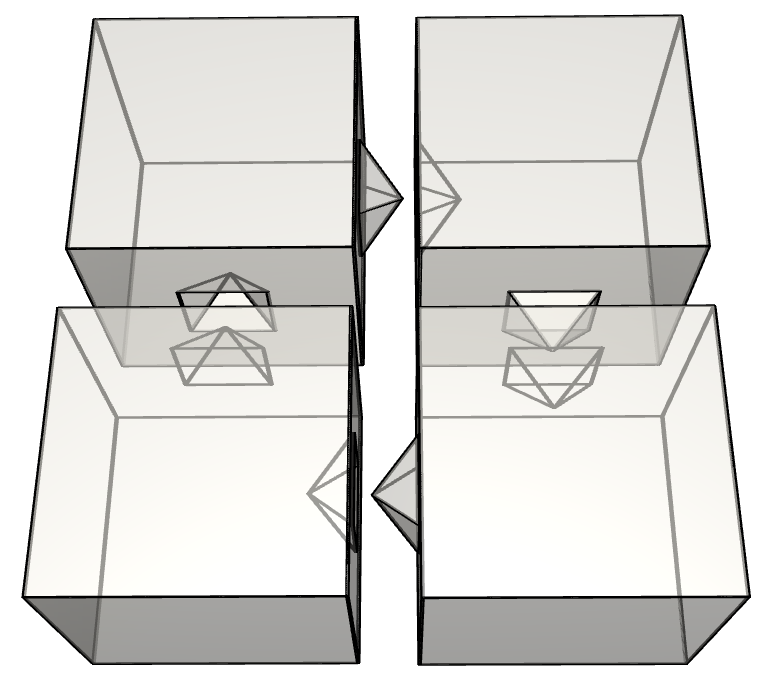}
        \subcaption{}
        \label{fig:modified_cube_assembly}
    \end{minipage}
    \caption{(a) Modified cube. (b) Exploded view of assembling four copies of the modified cube.}
\end{figure}

We can assemble this block in groups of four similarly to the cube assembly given in Figure \ref{fig:CubeAssembly}.
However, the assembly shown in Figure \ref{fig:ModifiedCubeAssembly} is not an interlocking assembly, as interlocking only occurs in groups of four blocks. The set consisting of the four gray blocks can simultaneously be moved without causing intersections by applying the admissible motion $[0,1]\to \mathrm{SE(3)},t\mapsto \left(\R^3\to \R^3, x\mapsto x + \left(0,0,t \right)^\intercal\right)$,  see Figure \ref{fig:ModifiedCubeAssemblyShift}.

\begin{figure}[H]
\centering
\begin{subfigure}{.3\textwidth}
  \centering
  \includegraphics[height=3.5cm]{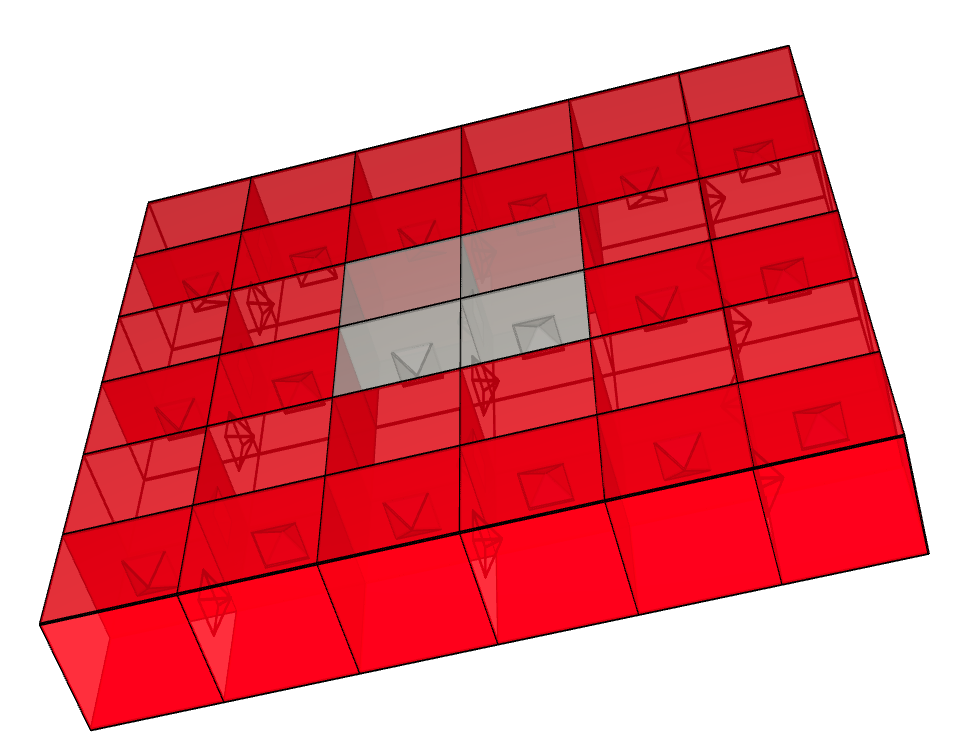}
  \subcaption{} \label{fig:ModifiedCubeAssembly}
\end{subfigure}%
\begin{subfigure}{.3\textwidth}
  \centering
  \includegraphics[height=3.5cm]{Interlocking_Figures/cube_deformed_assembly2.png}
  \subcaption{} \label{fig:ModifiedCubeAssembly2}
\end{subfigure}%
\begin{subfigure}{.3\textwidth}
  \centering
  \includegraphics[height=3.5cm]{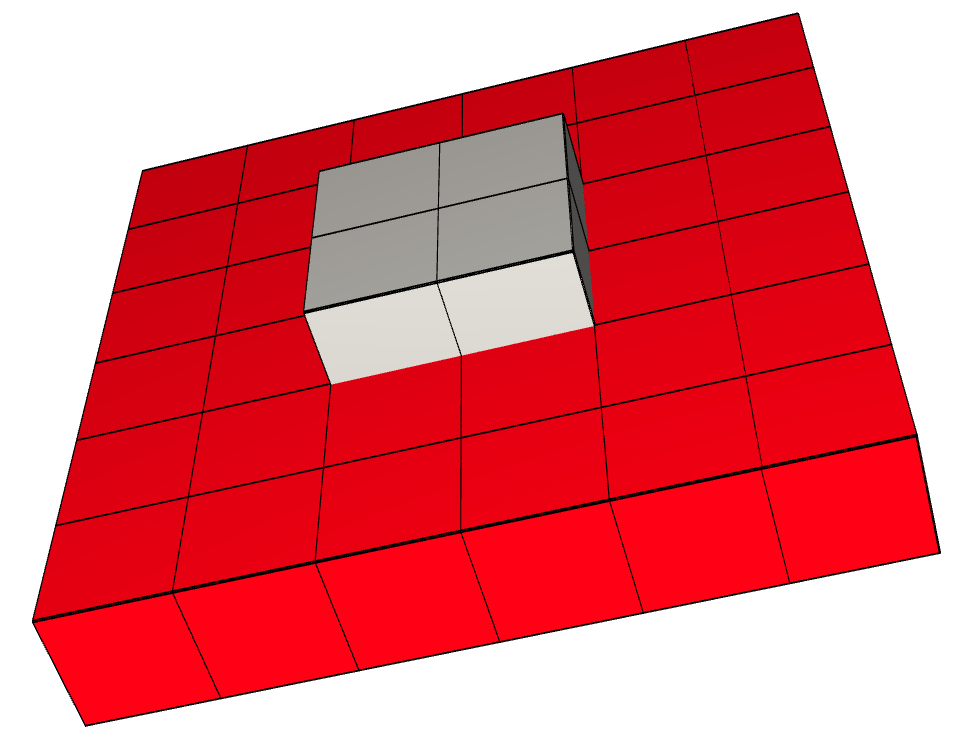}
  \subcaption{}
  \label{fig:ModifiedCubeAssemblyShift}
\end{subfigure}
\caption{(a,b) An assembly with modified cubes. (c) Even though each block is restrained by its neighbours, four blocks can be moved simultaneously.}
\label{fig:counter_example_interlocking}
\end{figure}

\subsection{The Definition of Interlocking Assemblies}

The following definition ensures that in an interlocking assembly, moving any finite subset of blocks leads to a violation of the assembly condition.

\begin{definition}\label{def:interlocking}
   An \emph{interlocking assembly} is an assembly of blocks $(X_i)_{i\in I}$ together with a subset $J\subset I$, called \emph{frame}, such that for all finite non-empty sets $\emptyset \neq T \subset I\setminus J$ and for all non-trivial admissible motions $(\gamma_i)_{i\in T}$ there exists $t\in [0,1]$ and $i,j\in I$ (set $\gamma_\ell\equiv \mathbb{I}$ if $\ell\notin T$ for $\ell\in \{i,j \}$) with 
\begin{equation}\label{eq:interlockingCondition}
    \gamma_i(t)(X_i) \cap \gamma_j(t)(X_j) \neq \partial \gamma_i(t)(X_i) \cap \partial \gamma_j(t)(X_j). \tag{$\bowtie$}
\end{equation}
If we restrict the motions to translations, we say that the assembly is a \emph{translational interlocking assembly}.
\end{definition}

 This definition is equivalent to saying that we cannot move a finite number of blocks not contained in the frame without causing intersections of blocks. 
% \begin{remark}
%     In some applications, the additional property is given that the midsections of the blocks yield a flat-surface which is motivated by flat vaults and in this case we call the assembly a \emph{topological interlocking assemblies}, see \cite{DyskinReview}.
     %If $|J|=1$ and $((X_i)_{i\in I},\emptyset)$ is not an interlocking assembly, we call the single element $j\in J$ the \emph{key}. 
%If the blocks in $J$ are connected such that joining them yields a torus, i.e.\$\bigcup_{j\in J}X_j$ is a torus, we call $((X_i)_{i\in I},J)$ a \emph{topological interlocking assembly}  (short \emph{TIA}). 
% \end{remark}

 \begin{remark}
The term \emph{topological interlocking assembly} is commonly used in the engineering and architecture literature to describe interlocking assemblies with a peripheral frame or ``force'' holding together the blocks, see \cite{EstrinDesignReview}. In a mathematical context, this terminology may lead to confusion, as it implies a topological classification (e.g., spheres or tori based on their genus) that diverges from the primarily geometric nature of these assemblies. To avoid such confusion and ensure clarity, we employ the term \emph{interlocking assembly}, omitting the prefix topological. This decision puts an emphasis on the geometric aspects of these structures and aims to avoid the potential confusion between topological and geometric concepts.

Moreover, the concept of \emph{interlocking puzzles} is also compatible with our definition above. Here, a special block in the assembly labelled as the ``key'' is used to lock the whole assembly, i.e. the key viewed as an element inside the block together with any other block yields the frame of an interlocking assembly as given in Definition \ref{def:interlocking}.

The concepts of topological interlocking and interlocking puzzles both deal with the potential for assembly and disassembly. Specifically, if the frame is no longer fixed, the parts of the assembly can often be taken apart. This idea is connected to \emph{assembly planning}, which focuses on the possibility of moving blocks from a starting position to an ending position without causing any penetrations, as outlined in \cite{wilson_geometric_1992}.
\end{remark}

\section{Infinitesimal Interlocking Criterion}\label{sec:infinitesimal}

In this section, we build on the observation that the definition of interlocking assemblies incorporates admissible motions which are differentiable maps of the form
$$\gamma \colon [0,1]\to \mathrm{SE}(3),\gamma(0)=\mathbb{I}$$
and we can differentiate $\gamma$ in $0$ to obtain an element in the Lie algebra $\mathfrak{se}(3)$, also known as the algebra of infinitesimal motions acting on $\R^3$, see Remark \ref{rem:lie_algebra}. This action can be exploited in order to give a linearised version of Definition \ref{def:interlocking} based on infinitesimal motions, which turns out to be the one given in \cite{wang_computational_2021}. For this, we consider two blocks $X_i,X_j$ inside an assembly of blocks with polyhedral boundary $(X_i)_{i\in I}$ and assume that the common boundary of the two blocks, i.e.\ $\partial X_i \cap  \partial X_j$ can be triangulated by contact triangles. For each contact triangle $f$ given by three vertices, we compute a normal vector $n$ pointing towards the block $X_j$. Let $p$ be one of the vertices of the given contact triangle $f$. When given two admissible motions $\gamma_i,\gamma_j$ for the blocks $X_i,X_j$, we can differentiate them in $0$ to obtain elements in $\Dot{\gamma_i}(0)=(\omega_i,t_i),\Dot{\gamma_j}(0)=(\omega_j,t_j)\in \mathfrak{se}(3)$ and act with them on the point $p$. In Proposition \ref{proposition:infinitesimal}, it is shown that the interlocking criterion \eqref{eq:interlockingCondition} translates into the inequality
$$
\left( (\omega_j,t_j).p-(\omega_i,t_i).p \right) \cdot n = \left( \omega_j \times p + t_j - (\omega_i \times p + t_i) \right) \cdot n \geq 0,
$$
which, using the rule that for all $\omega, p \in \mathbb{R}^3$, $(\omega \times p) \cdot n = (p \times n) \cdot \omega$, can be equivalently formulated as
\begin{equation}\label{ineq:infinitesimal_ineq_base}
    \left( (-p \times n)^\intercal, -n^\intercal, (-p \times n)^\intercal, -n^\intercal \right) \cdot \begin{pmatrix} \omega_i \\ t_i \\ \omega_j \\ t_j \end{pmatrix} \geq 0.
\end{equation}
The system of inequalities of the form above for a whole assembly are given in Definition \ref{def:infinitesimal_interlocking} and in Proposition \ref{proposition:infinitesimal}, we show that these inequalities indeed enforce the interlocking property, as given in Definition \ref{def:interlocking}.
In Figure \ref{fig:infinitesimal_crit}, we see a schematic illustration of this approach of modelling face to face contact using infinitesimal motions. 

\begin{figure}[H]
    \centering
    \scalebox{0.4}{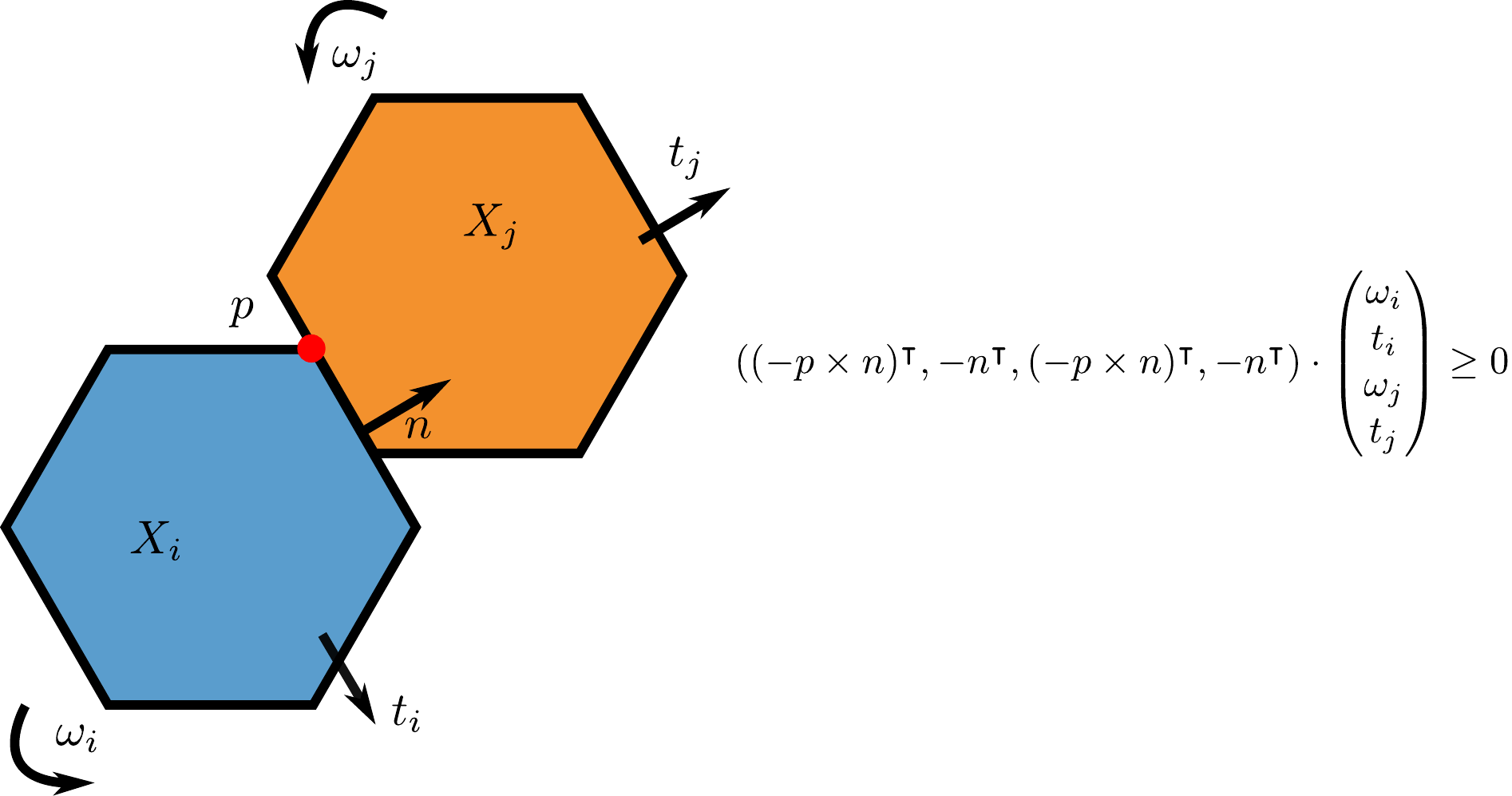}
    \caption{Schematic figure for interlocking criteria (after \cite{wang_computational_2021}): For two blocks $X_i,X_j$ and a contact point $p$ at a contact face with normal $n$ pointing towards the block $X_j$, we can derive an inequality corresponding to a non-penetration rule. Here, the motions of each block are associated with a tuple of infinitesimal motions, consisting of an angular momentum $\omega$ and a translation $t$.}
    \label{fig:infinitesimal_crit}
\end{figure}

\begin{remark}
    In \cite{wang_design_2019} an alternative infinitesimal interlocking criteria for convex blocks is presented. The idea is to replace the contact point $p$ in Inequality \eqref{ineq:infinitesimal_ineq_base} by $p-c_i$ and $p-c_j$, where $c_i$ and $c_j$ are the centres of the two convex blocks  $X_i$ and $X_j$, respectively. Moreover, it is shown in \cite{wang_design_2019} that this criterion can be translated into an equilibrium analysis using Farkas lemma, which gives a connection between linear inequalities and linear equalities. This approach of relying on the centres of blocks is generalised to non-convex blocks in \cite{stuttgen_modular_2023} by subdividing non-convex blocks into smaller convex blocks. In this work, we focus on the definition of infinitesimal interlocking assemblies as presented in \cite{wang_computational_2021}, which does not use the centres of each block.  This method can be also used for assemblies with non-convex blocks, and we establish a connection to the definition of interlocking assemblies in Proposition \ref{proposition:infinitesimal}.
\end{remark}

Since, we also consider infinite assemblies, we first define matrix-vector multiplications of infinite matrices.

\begin{remark}
    Let $\left(a_{ij}\right)=A\in\mathbb{R}^{\mathbb{N}\times\mathbb{N}}$ be an infinite dimensional matrix such that for all $i\in\mathbb{N}$ the set $\{j\mid a_{ij}\neq0\}$ is finite, i.e.\ in each row there are only finitely many non-zero entries. Then we can define the matrix-vector multiplication with a vector $x\in\mathbb{R}^{\mathbb{N}}$ in the usual way. In some cases, we restrict to vectors with finite support. Here, we write $\mathrm{supp}\left(x\right)=\{i\mid x_{i}\neq0\}$ and say that $x$ has finite support if $\lvert\mathrm{supp}\left(x\right) \rvert<\infty$.
\end{remark}

The following definition, based on the work \cite{wang_computational_2021},  establishes a connection between an assembly of blocks with triangulated polyhedral boundary and a set of inequalities described by a matrix modelling face-to-face contacts.

\begin{definition}\label{def:infinitesimal_interlocking}
    Let $X=(X_i)_{i\in I}$ be an assembly of polyhedra with triangulated polyhedral boundary, such that if the intersection of two blocks at their boundary has area larger than zero it is already given by common triangular faces. Let $J\subset I$ be a subset of $I$ (possibly empty). Let $C$ be the number of contact triangles between blocks inside the assembly, which are each defined by three vertices. Then we define the \emph{infinitesimal interlocking matrix} $A_{X,J}$ with $3\cdot C$ rows indexed by the defining vertices of all contact triangles and $6\cdot \lvert I\setminus J\rvert$ columns corresponding to possible admissible motions $(\omega,t)\in \R^6$ for each block not contained in the frame. For each contact triangle $F$ of blocks $i,j\in I\setminus J$, where $p$ is a defining vertex of $F$ and $n$ is a normal vector of $F$ pointing towards the block $j$, we obtain a row of $A_{X,J}$ of the form $$\big(0,\dots,\underbrace{(-p\times n)^\intercal,-n^\intercal}_{i},0,\dots,0,\underbrace{(p\times n)^\intercal,n^\intercal}_{j},0,\dots,0\big).$$For a block $i\in I\setminus J$ having contact with a block $j \in J$, we obtain a row of the form $$\big(0,\dots,\underbrace{(-p\times n)^\intercal,-n^\intercal}_{i},0,\dots,0,\dots,0\big).$$ The assembly is called \emph{infinitesimally interlocked} with \emph{frame} $J$ if $$ A_{X,J}\cdot x \geq 0 \text{ implies } x=0,$$ for any $x$ with finite support (the inequality $\geq 0$ is understood componentwise). Here, we identify $x$ with the family $(\gamma_i)_{i\in I}$ of infinitesimal motions, where $\gamma_i \in \mathfrak{se}(3)$ is of the form $\gamma_i=(\omega_i,t_i)\in \R^6$. By considering identical rows only once, we can reduce the system into an equivalent system with fewer inequalities and call the resulting matrix the \emph{reduced infinitesimal interlocking matrix}. We define the assembly to be \emph{infinitesimal translational interlocked} if there are no translational admissible motions, i.e.\ we set $\omega_i=0$ for each block and thus only consider for each block $X_i$  infinitesimal motions of the form $(0,t_i)$. In this case, we only need one row for each face-to-face contact, as we no longer dependent on the points $p$. The \emph{infinitesimal interlocking polytope} is defined as the set of vectors $x$ with finite support and $A_{X,J} \cdot x\geq 0$ (componentwise equal or larger than $0$).  
\end{definition}

 Using the cross product in the definition above is tied to the three-dimensional case as the dimension of $\text{SE}(3)$ equals $3$, and for general values $n$ the dimension of $\text{SE}(n)$ equals $\frac{n(n-1)}{2}$. 

\begin{remark}
     Both the interlocking and infinitesimal interlocking definition can be generalised to any dimension. In the case of the infinitesimal interlocking definition, we have to adapt the definition of the rows of the interlocking matrix by changing the part corresponding to infinitesimal rotations $\omega$ which is derived from the fact $(p\times n)\cdot \omega=(\omega \times p)\cdot n$, see Proposition \ref{proposition:infinitesimal}.
\end{remark}

Before showing that the infinitesimal definition implies the regular definition of an interlocking assembly, we give several examples of how to compute the interlocking matrix and showcase that even for assemblies with relatively few blocks, the matrix $A$ can be quite large. We start with the simple cube assembly given in Figure \ref{fig:CubeAssembly}, and show that its interlocking matrix has a non-trivial kernel with elements corresponding to admissible motions.

\begin{figure}[H]
\centering
\begin{minipage}{.33\textwidth}
  \centering
  \includegraphics[height=4cm]{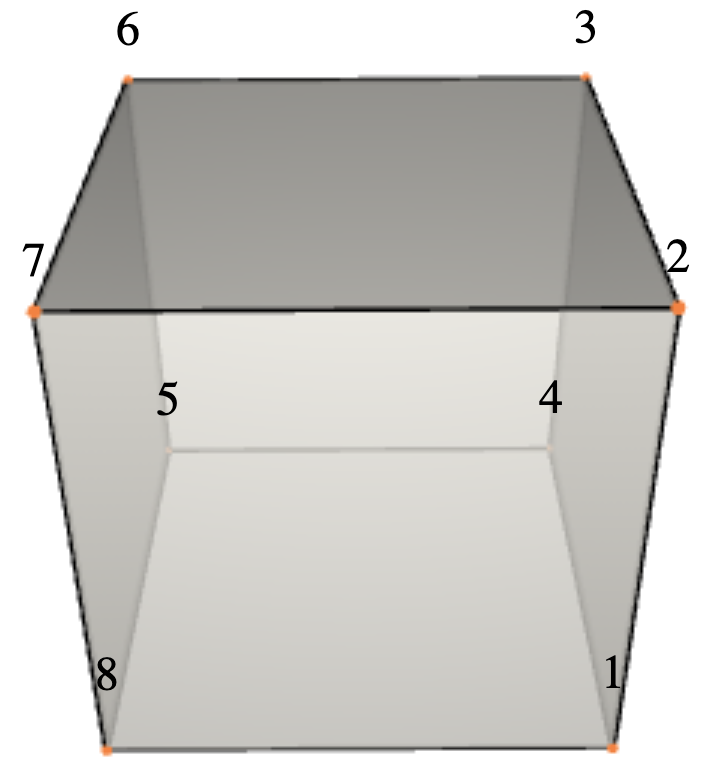}
  \subcaption{}
    \label{fig:Cube}
\end{minipage}%
\begin{minipage}{.33\textwidth}
  \centering
  \includegraphics[height=4cm]{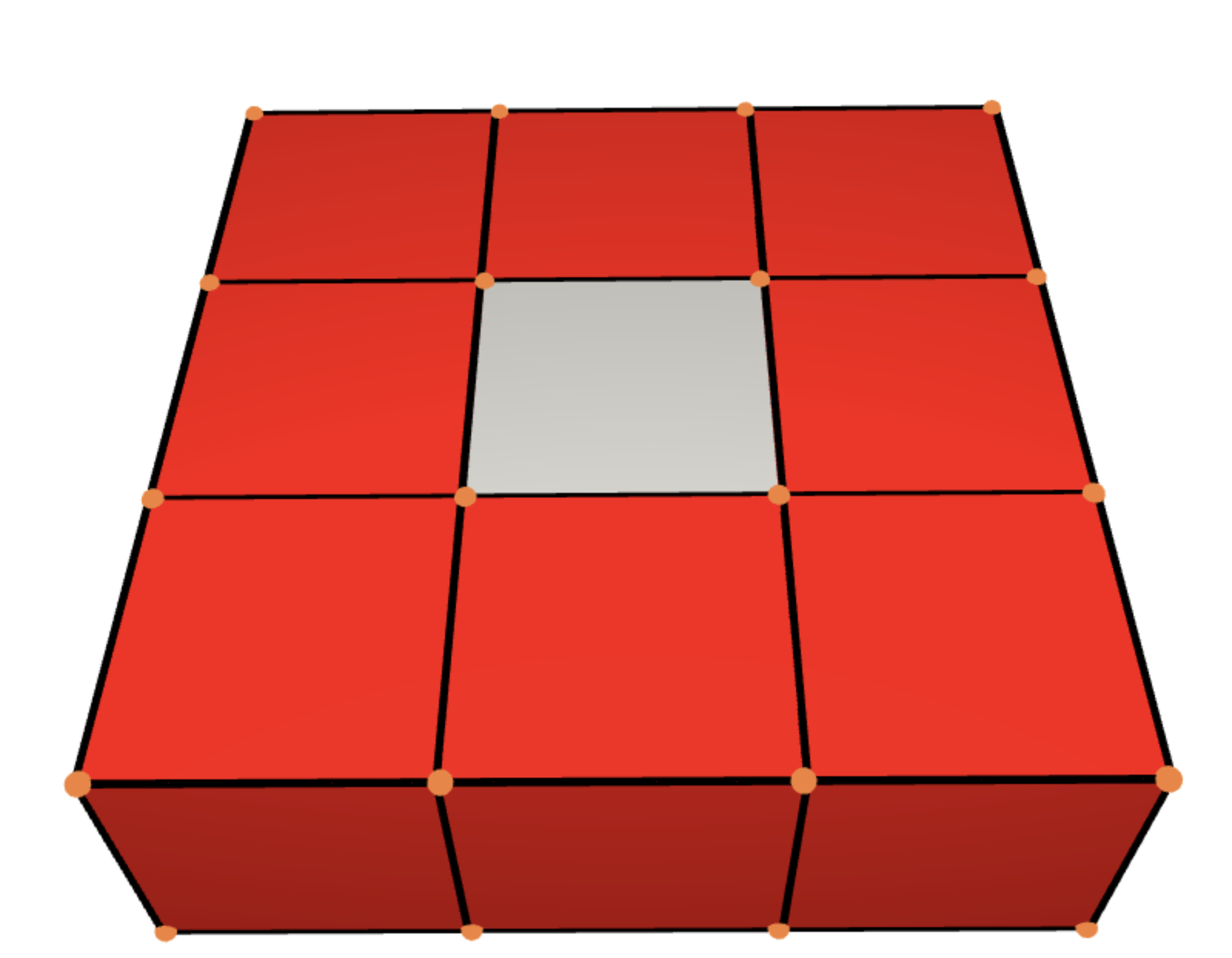}
  \subcaption{}
    \label{fig:CubeAssemblyNonInterlocking}
\end{minipage}
\begin{minipage}{.33\textwidth}
  \centering
  \includegraphics[height=4cm]{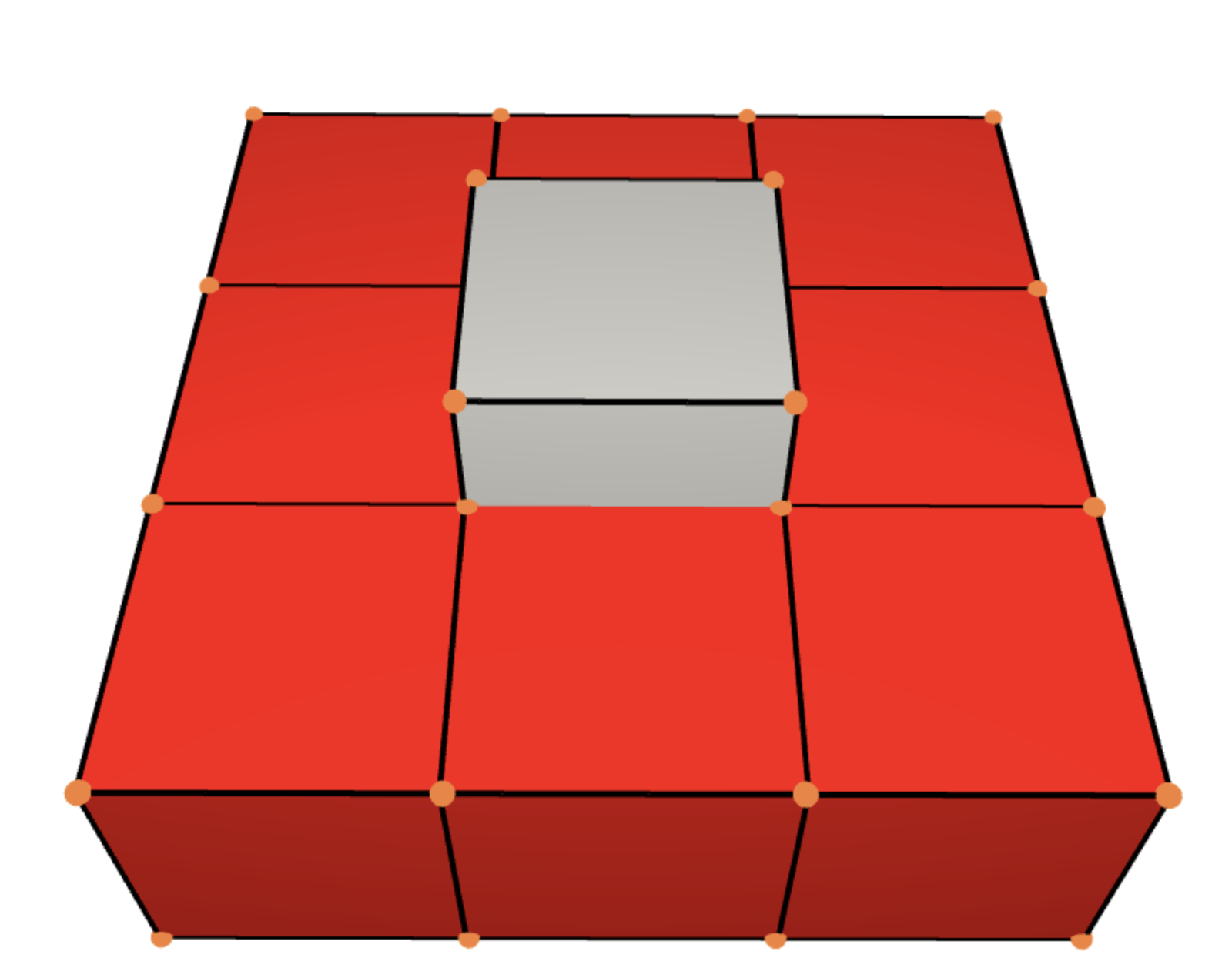}
  \subcaption{}
    \label{fig:SmallCubeAssemblyShift}
\end{minipage}
\caption{(a) Cube with vertex labels, (b) $3\times 3$ assembly of  cubes,  (c) $3\times 3$ assembly of  cubes with middle cube shifted by the translation $(0,0,\frac{1}{2})^\intercal$.}
\label{fig:SmallCubeAssemblyNonInterlocking}
\end{figure}

\begin{example}\label{example:cubes}
We consider a $3\times 3$ assembly of cubes as shown in Figure \ref{fig:CubeAssemblyNonInterlocking}. The outer cubes are fixed and only the inner cube is allowed to move. This cube in the middle can be obtained as the convex hull of the coordinates in the ordered list:
\[
[ \left( 0, 0, 1 \right)^\intercal, \left( 0, 0, 0 \right)^\intercal, \left( 0, 1, 0 \right)^\intercal, \left( 0, 1, 1 \right)^\intercal, \left( 1, 1, 1 \right)^\intercal, \left( 1, 1, 0 \right)^\intercal, \left( 1, 0, 0 \right)^\intercal, \left( 1, 0, 1 \right)^\intercal ]
\]
and the other cubes are obtained by translations of this cube in the directions $(1,0,0)$ and $(0,1,0)$. The contact-faces of the middle cube to its neighbouring cubes are given by the faces
\[
[1,2,3,4],[1,2,7,8],[3,4,5,6],[5,6,7,8],
\]
see Figure \ref{fig:Cube}. The corresponding outer normals of the faces above are given by
\[
\left(-1,0,0\right)^\intercal,\left(0,1,0\right)^\intercal,\left(0,-1,0\right)^\intercal,\left(1,0,0\right)^\intercal.
\]
By fixing a triangulation of the cube, we can use Definition \ref{def:infinitesimal_interlocking} to obtain the infinitesimal interlocking matrix with rows of the form $(-(p\times n)^\intercal,-n^\intercal)$, where $n$ is an outer normal and $p$ is a point belonging to a corresponding face. Note, that we only need four inequalities for each contact square, as a triangulation leads to redundant inequalities (some triangles share points and normals). Thus, we only obtain $16$ instead of $24$ rows and the reduced infinitesimal interlocking matrix is given as follows:
\[
A^\intercal= \begin{bsmallmatrix}
    0 & 0 & 0 & 0 & 1 & 0 & 0 & 1 & 0 & -1 & 0 & -1 & 0 & 0 & 0 & 0 \\
    1 & 0 & 0 & 1 & 0 & 0 & 0 & 0 & 0 & 0 & 0 & 0 & -1 & 0 & 0 & -1 \\
    0 & 0 & -1 & -1 & 0 & 0 & -1 & -1 & 0 & 0 & 1 & 1 & 1 & 1 & 0 & 0 \\
    1 & 1 & 1 & 1 & 0 & 0 & 0 & 0 & 0 & 0 & 0 & 0 & -1 & -1 & -1 & -1 \\
    0 & 0 & 0 & 0 & -1 & -1 & -1 & -1 & 1 & 1 & 1 & 1 & 0 & 0 & 0 & 0 \\
    0 & 0 & 0 & 0 & 0 & 0 & 0 & 0 & 0 & 0 & 0 & 0 & 0 & 0 & 0 & 0 \\
\end{bsmallmatrix}.
\]

The computation \(A \cdot (0,0,0,0,0,1)^\intercal = 0\) implies that an element of the form \((\omega,t)^\intercal = (0,0,0,0,0,1)^\intercal\) lies in the kernel of \(A\). This can be interpreted as shifting a cube outside the assembly along the translation $t=(0,0,1)^\intercal$, see Figure \ref{fig:ModifiedCubeAssemblyShift}. Indeed, we can use the map $\gamma:[0,1]\to \mathrm{SE(3)},t\mapsto \left(\R^3\to \R^3, x\mapsto x + \left(0,0,t \right)^\intercal\right)$ to shift out the cube in the middle, and it holds that $\Dot{\gamma}(0)=(0,0,1)^\intercal$. In general, we observe that for the matrix $A$ to have a trivial kernel, a necessary condition is that there have to be at least three normal vectors of faces $n_1,n_2,n_3$ with $\langle n_1,n_2,n_3 \rangle = \R^3.$
\end{example}

Next, we consider the assembly of stacked cubes as given in Figure \ref{fig:SmallCubeInterlocking}.

\begin{figure}[H]
\centering
\begin{minipage}{.33\textwidth}
  \centering
  \includegraphics[height=3cm]{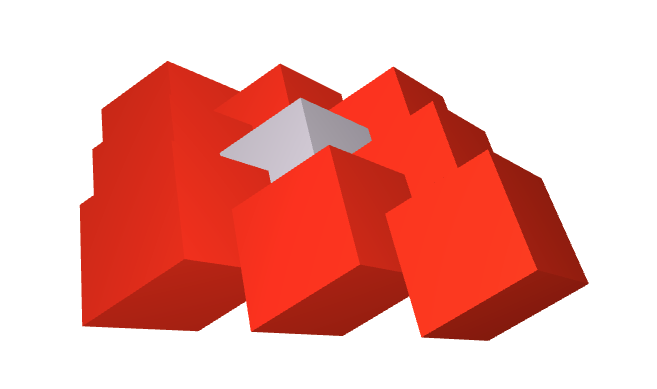}
  \subcaption{}
    \label{fig:SmallCubeAssembly}
\end{minipage}%
\begin{minipage}{.33\textwidth}
  \centering
  \includegraphics[height=3cm]{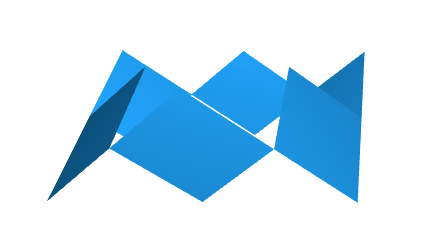}
  \subcaption{}
    \label{fig:SmallCubeAssemblyFaces}
\end{minipage}
\begin{minipage}{.33\textwidth}
  \centering
  \includegraphics[height=3cm]{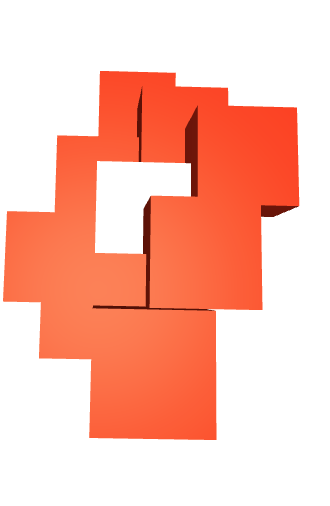}
  \subcaption{}%\captionof{figure}{Cube House}
    \label{fig:SmallCubeAssemblyOriented}
\end{minipage}
\caption{(a) $3\times 3$ Interlocking assembly of  cubes,  (b) contact faces of the middle cube with its neighbours, (c) orientation of cubes with integer coordinates.}
\label{fig:SmallCubeInterlocking}
\end{figure}

\begin{example}\label{example:cube_interlocking}

In order to describe the example shown in Figure \ref{fig:SmallCubeInterlocking} geometrically, we again consider the unit cube given by the points and faces, as in the previous example.

In order to embed the assembly shown in Figure \ref{fig:SmallCubeAssembly} in a way that the coordinates of all cubes are integers, we scale the unit cube by a factor $2$ such that its side length are all equal to $2$. Hence, we get the coordinates
$$
[ \left( 0, 0, 2 \right)^\intercal, \left( 0, 0, 0 \right)^\intercal, \left( 0, 2, 0 \right)^\intercal, \left( 0, 2, 2 \right)^\intercal, \left( 2, 2, 2 \right)^\intercal, \left( 2, 2, 0 \right)^\intercal, \left( 2, 0, 0 \right)^\intercal, 
  \left( 2, 0, 2 \right)^\intercal ].$$
Now, we obtain the assembly shown in Figure \ref{fig:SmallCubeAssemblyOriented} by applying the translations $v_1=(1,1,2)^\intercal,v_2=(-1,2,1)^\intercal$ to the cube above by translating the cube with $i \cdot v_1 + j \cdot v_2$ and $i,j\in \{0,1,2 \}$. The contact faces of the grey cube in the middle are shown in Figure \ref{fig:SmallCubeAssemblyFaces} and the resulting interlocking matrix determined by the contact faces is then given as follows:
\[ 
A^\intercal= \begin{bsmallmatrix}
   -3 & -3 & -4 & -4 & 4 & 4 & 5 & 5 & 0 & 0 & 0 & 0 & 0 & 0 & 0 & 0 & 4 & 4 & 3 & 3 & -5 & -5 & -4 & -4 \\
   0 & 0 & 0 & 0 & 0 & 0 & 0 & 0 & 3 & 3 & 4 & 4 & -4 & -4 & -5 & -5 & 0 & -1 & -1 & 0 & 1 & 2 & 2 & 1 \\
   1 & 2 & 2 & 1 & 0 & -1 & -1 & 0 & -5 & -4 & -4 & -5 & 4 & 3 & 3 & 4 & 0 & 0 & 0 & 0 & 0 & 0 & 0 & 0 \\
   0 & 0 & 0 & 0 & 0 & 0 & 0 & 0 & 1 & 1 & 1 & 1 & -1 & -1 & -1 & -1 & 0 & 0 & 0 & 0 & 0 & 0 & 0 & 0 \\
   1 & 1 & 1 & 1 & -1 & -1 & -1 & -1 & 0 & 0 & 0 & 0 & 0 & 0 & 0 & 0 & 0 & 0 & 0 & 0 & 0 & 0 & 0 & 0 \\
   0 & 0 & 0 & 0 & 0 & 0 & 0 & 0 & 0 & 0 & 0 & 0 & 0 & 0 & 0 & 0 & 1 & 1 & 1 & 1 & -1 & -1 & -1 & -1 \\
\end{bsmallmatrix}.
\]
Using a linear program solver such as the one available in \cite{scipy}, one can compute that the kernel of the matrix is trivial and there is no $0\neq x\in \R^6$ with $Ax\geq 0$.
%Removing one cube automatically results in a solution of the form $(0,0,0,t_1,t_2,t_3)$ for a certain infinitesimal translation $t\in \R^3$ 
Alternatively, we can centre the middle cube at the origin to receive the following matrix

\[ 
B^\intercal= \begin{bsmallmatrix}
    1 & 1 & 0 & 0 & 0 & 0 & 1 & 1 & 0 & 0 & 0 & 0 & 0 & 0 & 0 & 0 & 0 & 0 & -1 & -1 & -1 & -1 & 0 & 0 \\
    0 & 0 & 0 & 0 & 0 & 0 & 0 & 0 & -1 & -1 & 0 & 0 & 0 & 0 & -1 & -1 & 1 & 0 & 0 & 1 & 0 & 1 & 1 & 0 \\
    0 & 1 & 1 & 0 & 1 & 0 & 0 & 1 & -1 & 0 & 0 & -1 & 0 & -1 & -1 & 0 & 0 & 0 & 0 & 0 & 0 & 0 & 0 & 0 \\
    0 & 0 & 0 & 0 & 0 & 0 & 0 & 0 & 1 & 1 & 1 & 1 & -1 & -1 & -1 & -1 & 0 & 0 & 0 & 0 & 0 & 0 & 0 & 0 \\
    1 & 1 & 1 & 1 & -1 & -1 & -1 & -1 & 0 & 0 & 0 & 0 & 0 & 0 & 0 & 0 & 0 & 0 & 0 & 0 & 0 & 0 & 0 & 0 \\
    0 & 0 & 0 & 0 & 0 & 0 & 0 & 0 & 0 & 0 & 0 & 0 & 0 & 0 & 0 & 0 & 1 & 1 & 1 & 1 & -1 & -1 & -1 & -1 \\
\end{bsmallmatrix}.
\]

With this matrix, it is straightforward to see that a vector $x^\intercal=(\omega,t)$ with $Bx\geq 0$  has to satisfy $t=0$ and similarly $\omega=0$. Thus, the assembly is infinitesimally interlocked.
\end{example}

In the following example, we see how admissible motions translate into infinitesimal motions.

\begin{figure}[H]
\centering
\begin{minipage}{.3\textwidth}
  \centering
  \includegraphics[height=2.5cm]{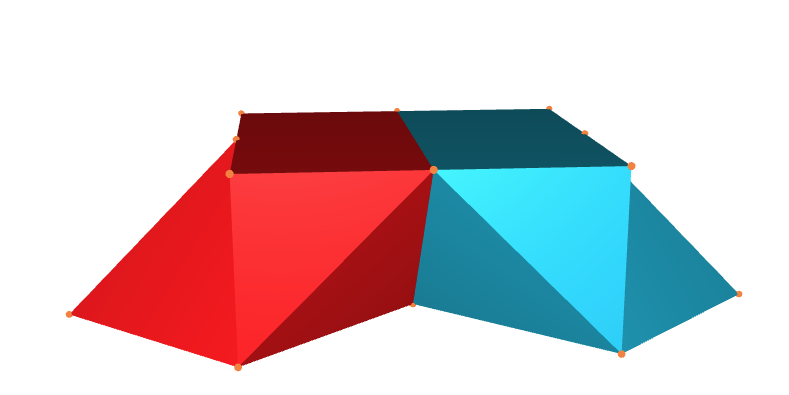}
  \subcaption{}
    \label{fig:TwoVersatileInitial}
\end{minipage}
\begin{minipage}{.3\textwidth}
  \centering
  \includegraphics[height=2.5cm]{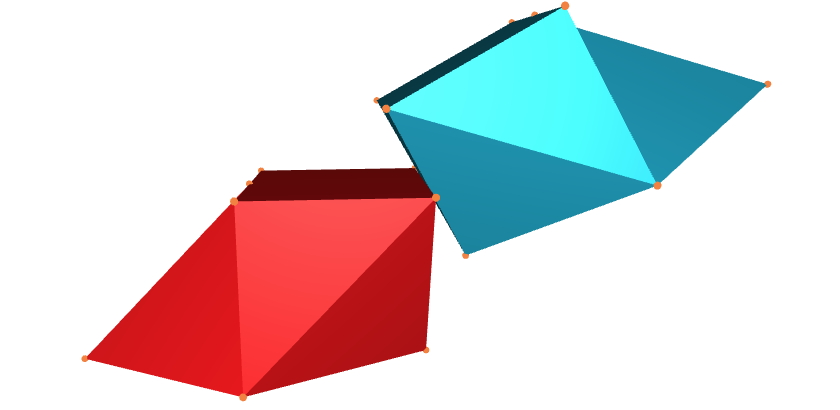}
  \subcaption{}
    \label{fig:TwoVersatileShift1}
\end{minipage}
\begin{minipage}{.3\textwidth}
  \centering
  \includegraphics[height=2.5cm]{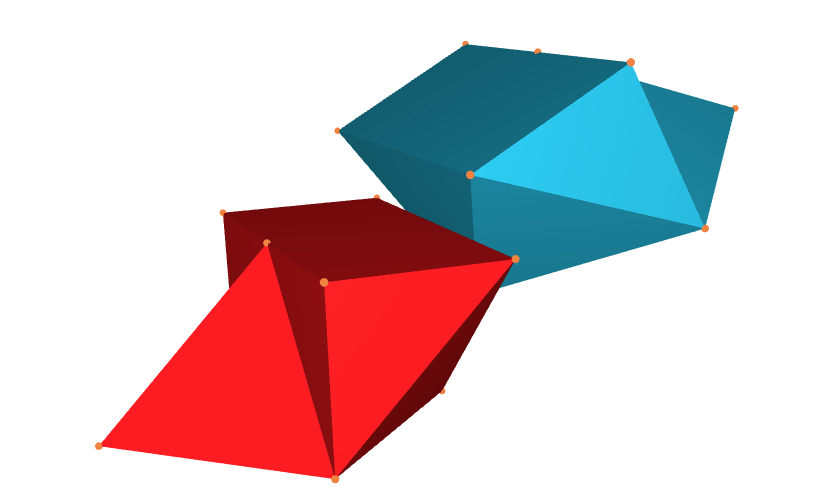}
  \subcaption{}
    \label{fig:TwoVersatileShift2}
\end{minipage}
\caption{Two \emph{Versatile Blocks} and application of an admissible motion: (a)  Initial placement (b,c) views of applying $\gamma(0.5)$ to the right block.}
\label{fig:TwoVersatile}
\end{figure}

\begin{example}\label{example:TwoVersatileBlocks}
    In this example, we consider two copies of a block, called \emph{Versatile Block}, first introduced in \cite{GoertzenFIB} and studied in \cite{GoertzenBridges}. We consider the two blocks shown in Figure \ref{fig:TwoVersatileInitial} with a contact triangle given by the points $\left(0,0,0\right)^\intercal,\left(0,-1,1\right)^\intercal,\left(0,1,1\right)^\intercal.$ The map
\begin{equation*}
\resizebox{\linewidth}{!}{
    $\gamma:[0,1]\to \mathrm{SE}(3),t\mapsto\left(\mathbb{R}^3 \to \mathbb{R}^3, x\mapsto\begin{pmatrix}\cos(t) & 0 & -\sin(t)\\
    0 & 1 & 0\\
    \sin(t) & 0 & \cos(t)
    \end{pmatrix}\cdot\left(x+\begin{pmatrix}0\\
    0\\
    t-1
    \end{pmatrix}\right)+\begin{pmatrix}0\\
    0\\
    1
    \end{pmatrix}\right)$
}
\end{equation*} applied to the right (blue) block in Figure \ref{fig:TwoVersatileInitial} rotates and shifts it along the upper edge of the contact triangle with the left (red) block, and can be also written as follows:
\[
\gamma\left(t\right)\left(x\right)=\begin{pmatrix}\cos\left(t\right) & 0 & -\sin\left(t\right) & -\sin\left(t\right)\cdot\left(t-1\right)\\
0 & 1 & 0 & 0\\
\sin\left(t\right) & 0 & \cos\left(t\right) & \cos\left(t\right)\cdot\left(t-1\right)+1\\
0 & 0 & 0 & 1
\end{pmatrix}\cdot x,
\]
where $x$ lies in the affine space $\mathrm{Aff}(\R^3)$. We see that $\gamma(0)=\mathbb{I}$ is the identity map and differentiating entrywise in $0$ yields:
\[
\dot{\gamma}\left(0\right)=\begin{pmatrix}0 & 0 & -1 & 1\\
0 & 0 & 0 & 0\\
1 & 0 & 0 & 2\\
0 & 0 & 0 & 0
\end{pmatrix}\in \mathfrak{se}(3).
\]
The map $\gamma$ is a continuous motion that applied to the right block in Figure \ref{fig:TwoVersatileInitial} does not lead to any intersections, assuming that the left block in Figure \ref{fig:TwoVersatileInitial} is restrained from moving. However, the points $\gamma(t)\left( \left(0,-1,1\right)^\intercal\right),\gamma(t)\left(\left(0,1,1\right)^\intercal\right)$ do not lie on the right side of the left block for any $t>0$. In order to show that this example is compatible with Definition \ref{def:infinitesimal_interlocking}, we have to show that for the points $p \in \{ \left(0,0,0\right)^\intercal,\left(0,-1,1\right)^\intercal,\left(0,1,1\right)^\intercal\}$ and the normal vector $n=\left(1,0,0\right)^\intercal$ of the contact triangle of the two blocks pointing towards the block on the right, the following inequality is satisfied
$$(p\times n,n)\cdot \dot{\gamma}\left(0\right) \geq 0,$$
where we identify $\dot{\gamma}\left(0\right)$ with $(0,-1,0,1,0,2)$.  Indeed, for $p=\left(0,-1,1\right)$ we have
\[
((p\times n)^\intercal,n^\intercal)\cdot \dot{\gamma}\left(0\right)=(0,1,1,1,0,0)^\intercal \cdot (0,-1,0,1,0,2)=0,
\]
for $p=\left(0,0,0\right)^\intercal$ we get 
\[
((p\times n)^\intercal,n^\intercal)\cdot \dot{\gamma}\left(0\right)=(0,0,0,1,0,0)^\intercal\cdot (0,-1,0,1,0,2)=1
\]
and for $p=\left(0,1,1\right)^\intercal$ we obtain
\[
((p\times n)^\intercal,n^\intercal)\cdot \dot{\gamma}\left(0\right)=(0,1,-1,1,0,0)^\intercal\cdot(0,-1,0,1,0,2)=0.
\]
\end{example}

In order to show that the infinitesimal interlocking property implies the usual interlocking property, we need to define the \emph{positive side} of a plane which depends on a fixed choice of a vector normal to the plane.

\begin{definition}\label{def:positive_side_plane}
    The \emph{positive side} of a plane $P=\left(n,d\right)$ given by a normal vector $n$ and a point $d\in \R^3$ which is contained in $P$ is defined as all points $p$ that are obtained by summing a point of the plane together with its normal multiplied by a non-negative factor, i.e. there exists $a\geq0$ such that $p=a\cdot n+d+v$ with $n\cdot v=0$.
\end{definition}

We can reformulate Definition \ref{def:positive_side_plane} as shown in the following lemma.

\begin{lemma}\label{lemma:plane_side}
    Let $P=\left(n,d\right)$ be a plane given by its normal vector $n$ and a point $d\in \R^3$ which is contained in $P$. Let $p\in\mathbb{R}^{3}$ be a point. Then $p$ is contained on the positive side of $P$ if and only if $\left(p-d\right)\cdot n\geq 0$. 
\end{lemma}
\begin{proof}
    The point $p$ is contained on the positive side of $P$ if and only if there exists $a\geq0$ such that $p=a\cdot n+d+v$ with $n\cdot v=0$. It follows that $p\cdot n=a\cdot\lVert n\rVert_{2}^{2}+d\cdot n\geq d\cdot n$ and thus $(p-d)\cdot n\geq0.$ The point $p$ lies on the other side of $P$, if $a<0$ and in this case we can show analogously that $(p-d)\cdot n\leq0.$
\end{proof}

As highlighted in Remark \ref{rem:lie_algebra}, the Lie algebra of the Lie group $\mathrm{SE}(3)$ is given by $\mathfrak{se}(3)$. The $6$-dimensional algebra $(\omega,t)\in \mathfrak{se}(3)$ acts on a point $p\in \R^3$ via $\mathrm{SE}(3)\times \R^3 \to \R^3, \omega \times p + t.$ The main difficulty of relating the infinitesimal definition to the definition of an interlocking assembly is translating the condition \eqref{eq:interlockingCondition} in Definition \ref{def:interlocking} into an infinitesimal version. %We restrict ourselves to a version that only considers face to face contact as it suffices for us, see Remark \ref{rem:only_face_contacts}.

\begin{proposition}\label{proposition:infinitesimal}
    Let $(X_i)_{i\in I}$ be an assembly together with a subset $J\subset I$. If for all finite subsets $\emptyset \neq T\subset I\setminus J$ and infinitesimal motions $x=(\gamma_i)_{i\in I}$, with $\gamma_i \in \R^6$ and $\gamma_i=0$ for all $i\in I\setminus T$ ($x$ has finite support) with $$A_{X,J} \cdot x \geq 0$$ implies that $x=0$, then $(X_i)_{i\in I}$ is an interlocking assembly with frame $J$.
\end{proposition}

\begin{proof}
If there is an admissible family of non-trivial motions $(\gamma_i)_{i\in T}$ for a non-empty finite subset of the blocks $T \subset I\setminus J$, we can assume that for at least one $i\in T$, we have that $\dot{\gamma_i}(0) \neq 0$.
Now, let $p$ be a contact point of a contact face with vertices $p,p',p''$ of the blocks $X_{i},X_{j}$ with normal $n_{i}$ pointing towards $X_{j}$. Since the motions $\gamma_i$ for $i\in T$ do not lead to penetrations, it follows that for all admissible motions $\gamma_{i},\gamma_{j}$ of the underlying assembly there either exists $\varepsilon>0$ such that for all $t\in[0,\varepsilon)$ the point $\gamma_j(t)(p)$ is contained in the positive side of the plane $P=\left(\gamma_{i}\left(t\right)\left(n_{i}\right),\gamma_{i}\left(t\right)\left(p\right)\right)$ and  with Lemma \ref{lemma:plane_side} this is equivalent to
\[
\left(\gamma_{j}\left(t\right)\left(p\right)-\gamma_{i}\left(t\right)\left(p\right)\right)\cdot\gamma_{i}\left(t\right)\left(n_{i}\right)\geq0,
\]
or if there is no such $\varepsilon$ (see Example \ref{example:TwoVersatileBlocks}), we instead consider the points $p-\frac{v}{k}$ for $k\in\mathbb{N}$, $v=\frac{\left(p'-p\right)+\left(p''-p\right)}{3}$ and get
\[
\left(\gamma_{j}\left(t\right)\left(p-\frac{v}{k}\right)-\gamma_{i}\left(t\right)\left(p-\frac{v}{k}\right)\right)\cdot\gamma_{i}\left(t\right)\left(n_{i}\right)\geq0,
\]
  for $t\in [0,\varepsilon_k)$ and take the limit $k\to\infty$. It follows that 
\begin{align*}0 & \leq\left(\gamma_{j}\left(t\right)\left(p\right)-\gamma_{i}\left(t\right)\left(p\right)\right)\cdot\gamma_{i}\left(t\right)\left(n_{i}\right)\\
 & =\left(\gamma_{j}\left(t\right)\left(p\right)-p-\left(\gamma_{i}\left(t\right)\left(p\right)-p\right)\right)\cdot\gamma_{i}\left(t\right)\left(n_{i}\right)
\end{align*}
which is equivalent to the following by multiplying with $\frac{1}{t}$ for $t>0$:
\begin{align*}
0 & \leq\frac{1}{t}\left(\gamma_{j}\left(t\right)\left(p\right)-p-\left(\gamma_{i}\left(t\right)\left(p\right)-p\right)\right)\cdot\gamma_{i}\left(t\right)\left(n_{i}\right)\\
 & =\left(\frac{\gamma_{j}\left(t\right)\left(p\right)-p}{t}-\frac{\gamma_{i}\left(t\right)\left(p\right)-p}{t}\right)\cdot\gamma_{i}\left(t\right)\left(n_{i}\right).
\end{align*}
Since $\gamma_{i},\gamma_{j}$ are differentiable and thus continuous we have $${\displaystyle \lim_{t\to0}}\gamma_{i}\left(t\right)\left(n_{i}\right)=n_i$$
and thus it follows that
\[
0\leq\left(\dot{\gamma_{j}}\left(0\right)\left(p\right)-\dot{\gamma_{i}}\left(0\right)\left(p\right)\right)\cdot n_{i}.
\]
For an admissible motion $\gamma$, the derivative at $0$ given by $\dot{\gamma}\left(0\right)$ is an element in the Lie algebra $\mathfrak{se}\left(3\right)$ and can be represented as a $6-$dimensional vector $\left(\omega,t\right)$, see Remark \ref{rem:lie_algebra}. 
%Such an element can be interpreted as an infinitesimal motion with $\omega$ an \emph{angular momentum} and $t$ a translation. Such an element acts on $\mathbb{R}^{3}$ via \[ \mathbb{R}^{3}\to\mathbb{R}^{3},p\mapsto\omega\times p+t, \] with $\times$ denoting the cross product. 
If $\dot{\gamma_{j}}\left(0\right)$ corresponds to $\left(\omega_{j},t_{j}\right)$ and $\dot{\gamma_{i}}\left(0\right)$ to $\left(\omega_{i},t_{i}\right)$ we get
\[
0\leq\left(\omega_{j}\times p+t_{j}-\left(\omega_{i}\times p+t_{i}\right)\right)\cdot n_{i}.
\]
For vectors $a,b,c\in \R^3$, we have that 
\[
\left(a\times b\right)\cdot c=\left(b\times c\right)\cdot a,
\]
since the determinant of the matrix $(a,b,c)\in \R^{3\times 3}$  can be expressed as $\left(a\times b\right)\cdot c$ and thus it follows that
\[
\left(a\times b\right)\cdot c=\det\left(a,b,c\right)=\det\left(b,c,a\right)=\left(b\times c\right)\cdot a.
\]
With this we can compute the following:
\begin{align*}0 & \leq\left(\omega_{j}\times p+t_{j}-\left(\omega_{i}\times p+t_{i}\right)\right)\cdot n_{i}\\
 & =\omega_{j}\times p\cdot n_{i}+t_{j}\cdot n_{i}-\omega_{i}\times p\cdot n_{i}-t_{i}\cdot n_{i}\\
 & =p\times n_{i}\cdot\omega_{j}+t_{j}\cdot n_{i}-p\times n_{i}\cdot\omega_{i}-t_{i}\cdot n_{i}\\
 & =-p\times n_{i}\cdot\omega_{i}-n_{i}\cdot t_{i}+p\times n_{i}\cdot\omega_{j}+n_{i}\cdot t_{j}\\
 & =\left(-p\times n_{i},-n_{i},p\times n_{i},n_{i}\right)\cdot\left(\omega_{i},t_{i},\omega_{j},t_{j}\right).
\end{align*}
In total, we conclude that
\[
0\leq\left(-p\times n,-n,p\times n ,n\right)\cdot\left(\omega_{i},t_{i},\omega_{j},t_{j}\right).
\]
It follows that we obtain the infinitesimal interlocking matrix $A_{X,J}$ of the underlying assembly in this way. If we assume that there is no $x\neq 0 $ with finite support such that $A_{X,J}\cdot x \geq 0$, then there is no non-trivial admissible motion $\gamma$ since we enforce that for such motions the corresponding infinitesimal motion given by the derivate of $\gamma$ in $0$ is non-trivial, i.e.\ $\Dot{\gamma}(0)\neq 0$.
\end{proof}

Given this  infinitesimal criterion, the question arises how it can be used to prove the interlocking property given in Definition \ref{def:interlocking}. In general, we proceed as follows with a given interlocking matrix $A$:
    \begin{enumerate}
        \item Show that for any $x$, the existence of a row index $i$ with $(Ax)_i>0$ implies that $\lvert \mathrm{supp}(x) \rvert=\infty$.
        \item Show that the kernel of $A$ is trivial.
    \end{enumerate}
    This proves that the only admissible infinitesimal motion is the zero vector.

In general, we need to consider not only inequalities arising from face-to-face contacts.

\begin{remark}\label{rem:only_face_contacts}
    In Definition \ref{def:infinitesimal_interlocking}, we only consider contacts of face pairs. In Section \ref{sec:proof}, we show that this suffices to prove the interlocking property for certain assemblies. In general, restricting to face pairs only is not sufficient to prove the interlocking property, see \cite{wang_design_2019,stuttgen_modular_2023}. In these situations one needs to model further contact types leading to further inequalities, i.e.
    \begin{itemize}
        \item vertex-vertex contact,
        \item vertex-edge contact,
        \item vertex-face contact,
        \item edge-edge contact,
        \item edge-face contact,
        \item face-face contact.
    \end{itemize}
    Considering these additional contact types increases the complexity of the problem immensely. In certain situations, we can simplify these contact relations using symmetries of the underlying assembly.
\end{remark}

\section{Interlocking Property of Assemblies with Wallpaper Symmetry}\label{sec:proof}

In this section, we establish interlocking properties of assemblies constructed with the methods based on continuously deforming fundamental domains of a given wallpaper group $G$ into each other, as presented in the previous section. From these infinite assemblies, we can also obtain interlocking assemblies with finitely many blocks by only considering finite portions, see Corollary \ref{corollary:finite_interlocking}. For this purpose, we first show that the infinitesimal interlocking property (Definition \ref{def:infinitesimal_interlocking}) of the infinite assemblies construction in \cite{goertzen2024constructing} can be decoded into an infinite dimensional polytope carrying the same symmetries as the underlying assembly. This observation allows a simplification for proving the interlocking property. Next, we show that no ``sliding'' motions are possible for any constructed assemblies satisfying certain conditions.

\begin{definition}
    Let \((X_i)_{i \in I}\) be an assembly of blocks and \(\gamma = (\gamma_i)_{i \in I}\) be admissible motions such that \(\gamma_i \equiv 0\) for almost all \(i \in I\). We say that \(\gamma\) consists of \emph{sliding motions} if the infinitesimal interlocking matrix \(A\) of the assembly \((X_i)_{i \in I}\) satisfies $$ Ax = 0,$$
where \(x = (\dot{\gamma}_i(0))_{i \in I}\).
\end{definition}

In Example \ref{example:cubes} we show that assembling cubes can lead to sliding motions which can be viewed as motions such that contact faces remain in contact when applying the motion. In Proposition \ref{proposition:kernel}, we establish that the assemblies constructed in the previous section do not admit any sliding motion if we assume sufficient edge deformations specified as Criterion \ref{criterion:interlocking}.

Using this, we show that we can classify assemblies with a translational interlocking property and p1 symmetry using the notion of ``infinite interlocking chains''.

\begin{definition}
    Let $(X_i)_{i\in I}$ be an assembly of blocks, $A$ its infinitesimal interlocking matrix and let $\gamma=(\gamma_i)_{i\in I}$ be admissible motions. The infinitesimal motion $x=(\Dot{\gamma_i}(0))_{i\in I}$ is an \emph{infinite interlocking chain} if $Ax\geq 0$ and $x$ has infinite support.
\end{definition}

Consider an assembly coming from the construction in \cite{goertzen2024constructing} by continuously deforming a fundamental domain $F$ of a wallpaper group $G$  into another domain $F'$ of the same group leading to a block $X$ that can be assembled using the extended action of $G$ onto $\R^3$, i.e.\ we consider the infinite assembly of blocks $(X_g)_{g\in G}=(g(X))_{g\in G}$. If there exists a finite set $H\subset G$ such that $H(F)=H(F')$ it follows that we can simultaneously shift the blocks corresponding to $H$ upwards using the admissible motion $\gamma \colon t\mapsto \big( x\mapsto (x_1,x_2,x_3+t) \big)$. This leads to the following interlocking criterion.

\begin{criterion}\label{criterion:interlocking}
    The assembly $(X_g)_{g\in G}$ is a translational interlocking assembly with empty frame if and only if there is no finite set $\emptyset \neq H\subset G$ such that $\bigcup_{h\in H}h(F)=\bigcup_{h\in H}h(F')$. This is equivalent to saying that for each finite set $\emptyset \neq H\subset G$ the edges on the boundary of $\bigcup_{h\in H}h(F)$ are deformed.
\end{criterion}

Indeed, we prove the following for wallpaper groups of type p1 in Theorem \ref{thm:p1_translational}.

\begin{theorem*}
    If $G$ is of type p1, the assemblies $(X_g)_{g\in G}$ are translational interlocked if and only if Criterion \ref{criterion:interlocking} holds.
\end{theorem*}

\subsection{A Polytope with Wallpaper Symmetry}

Let $X$ be a block coming from the construction in \cite{goertzen2024constructing} which can be assembled using the action of a planar crystallographic group $G$. Hence, there is an assembly of the form $(X_g)_{g\in G}$ with $X=X_{\mathbb{I}}$, where $\mathbb{I}$ is the identity element of $G$. In the construction of $X$ we consider edge representatives $e_1,\dots,e_n \subset \R^2$ of an initial fundamental domain $F$ together with intermediate points $p_1,\dots,p_n \in \R^2$ determining a triangulation of the surface of $X$. For a given edge $e_i$ of $F$, we distinguish between the two cases $p_{i-1}\neq p_i\neq p_{i+1}$ and $p_{i-1}= p_i$ or $p_i= p_{i+1}$.

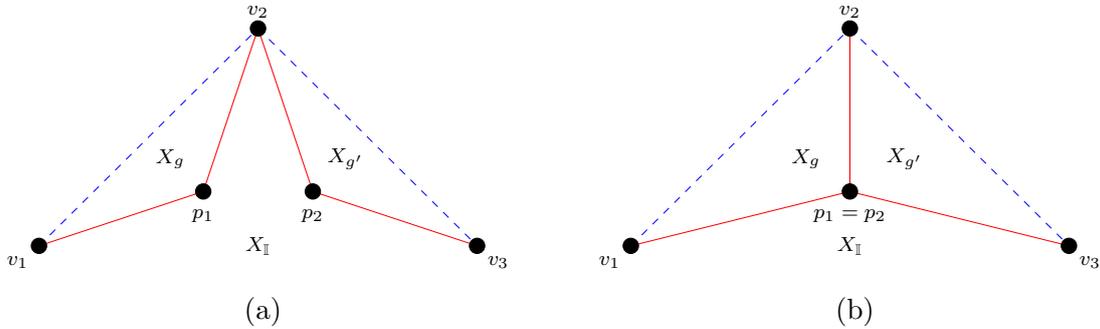
\begin{figure}[H]
    \centering
    \begin{minipage}{.45\textwidth}
        \centering
        \resizebox{7cm}{!}{\begin{tikzpicture}[scale=3]

% Vertices
\coordinate (v1) at (0,0);
\coordinate (v2) at (1,1);
\coordinate (v3) at (2,0);
\coordinate (p1) at (0.75,0.25);
\coordinate (p2) at (1.25,0.25);

% Edges
\draw[blue,dashed] (v1) -- (v2);
\draw[blue,dashed] (v2) -- (v3);

% Paths
\draw[red] (v1) -- (p1) -- (v2);
\draw[red] (v2) -- (p2) -- (v3);

% Drawing points
\filldraw [black] (v1) circle (1pt) node[anchor=north east] {\scriptsize $v_1$};
\filldraw [black] (v2) circle (1pt) node[anchor=south] {\scriptsize $v_2$};
\filldraw [black] (v3) circle (1pt) node[anchor=north west] {\scriptsize $v_3$};
\filldraw [black] (p1) circle (1pt) node[anchor=north,yshift=-0.1cm] {\scriptsize $p_1$};
\filldraw [black] (p2) circle (1pt) node[anchor=north,yshift=-0.1cm] {\scriptsize $p_2$};

% Labels inside triangles
\node at (0.6,0.4) {\scriptsize $X_g$};
\node at (1.4,0.4) {\scriptsize $X_{g'}$};
\node at (1,0) {\scriptsize $X_{\mathbb{I}}$};

\end{tikzpicture}}
        \subcaption{}
        \label{fig:contact_case_1}
    \end{minipage}
    \begin{minipage}{.45\textwidth}
        \centering
        \resizebox{7cm}{!}{\begin{tikzpicture}[scale=3]

% Vertices
\coordinate (v1) at (0,0);
\coordinate (v2) at (1,1);
\coordinate (v3) at (2,0);
\coordinate (p1) at (1,0.25);
\coordinate (p2) at (1,0.25);

% Edges
\draw[blue,dashed] (v1) -- (v2);
\draw[blue,dashed] (v2) -- (v3);

% Paths
\draw[red] (v1) -- (p1) -- (v2);
\draw[red] (v2) -- (p2) -- (v3);

% Drawing points
\filldraw [black] (v1) circle (1pt) node[anchor=north east] {\scriptsize $v_1$};
\filldraw [black] (v2) circle (1pt) node[anchor=south] {\scriptsize $v_2$};
\filldraw [black] (v3) circle (1pt) node[anchor=north west] {\scriptsize $v_3$};
\filldraw [black] (p1) circle (1pt) node[anchor=north,yshift=-0.1cm] {\scriptsize $p_1=p_2$};

% Labels inside triangles
\node at (0.8,0.4) {\scriptsize $X_g$};
\node at (1.25,0.4) {\scriptsize $X_{g'}$};
\node at (1,0) {\scriptsize $X_{\mathbb{I}}$};

\end{tikzpicture}}
        \subcaption{}
        \label{fig:contact_case_2}
    \end{minipage}
    \caption{We can deform two neighbouring edges with vertices $\{v_1,v_2 \}$ and $\{v_2,v_3 \}$ of an initial fundamental domain $F$ by introducing intermediate points $p_1,p_2$. The choices of $p_1,p_2$ determine the contact of the resulting block $X=X_{\mathbb{I}}$ to other blocks $X_g,X_{g'}$ inside the assembly. (a) If $p_1\neq p_2$ then the resulting triangles associated to the edges meet with a single block. (b) If $p_1=p_2$, we obtain a contact triangle between the blocks $X_g,X_{g'}$.}
    \label{fig:enter-label}
\end{figure}

The contact faces of the block $X$ to the neighbouring blocks can be determined by considering each edge independently. For an edge with intermediate point satisfying the first case above (see Figure \ref{fig:contact_case_1}), we obtain rows in the interlocking matrix (Definition \ref{def:infinitesimal_interlocking}) of the following form, since the resulting triangles all are in contact with a single block $X_g$:

\[
\begin{blockarray}{cc}
X & X_g \\
\begin{block}{[cc]}
  * & *  \\
  * & *  \\
  * & *  \\
\end{block}
\end{blockarray}
\]

This means that the three triangles coming from the deformation of the given edge are all in contact with another block. 
For the special case $p_i=p_{i+1}$ (see Figure \ref{fig:contact_case_2}) we have fewer faces and for each edge $e_i$, there are two blocks $X_g,X_{g'}$ that intersect with $X$ at the faces defined at $e_i$ leading to a matrix of the form

\[
\begin{blockarray}{ccc}
X & X_g & X_{g'} \\
\begin{block}{[ccc]}
  * & * & 0 \\
  * & * & 0 \\
  0 & * & * \\
\end{block}
\end{blockarray}.
\]

Each block contact is contained in the matrices of the form above. Since we assumed that the block $X$ is constructed with the techniques given in the previous section, its boundary is given by the triangulation $X_{F,F'}$, where $F,F'$ are two fundamental domains. For each vertical or tiled face $T$ with normal vector $n$ of $X_{F,F'}$  that is in contact with another block in the assembly corresponding to a group element $g\in G$, we can find an edge representative $e\in \{e_1,\dots,e_m \}$ of $F$ such that $T$ is obtained w.l.o.g. by the deformation of $e$ as in Figure \ref{fig:deforming_paths}.

\begin{figure}[H]
    \centering
     \begin{minipage}[b]{.25\textwidth}
        \centering
        \begin{tikzpicture}
            % Define points
            \coordinate (v1) at (0,0);
            \coordinate (v2) at (2,0);
    
            % Draw the path
            \draw[->] (v1) -- (v2);
    
            % Draw points
            \foreach \point in {v1,v2}
                \draw[fill=black] (\point) circle (2pt);
    
            % Label points
            \node[anchor=east] at (v1) {$v_1$};
            \node[anchor=west] at (v2) {$v_2$};
         \end{tikzpicture}
        \subcaption{}
    \end{minipage}
     \begin{minipage}[b]{.25\textwidth}
        \centering
        \begin{tikzpicture}
            % Define points
            \coordinate (v1) at (0,0);
            \coordinate (p) at (1,1);
            \coordinate (v2) at (2,0);
    
            % Draw the path
            \draw[->] (v1) -- (p) -- (v2);
    
            % Draw points
            \foreach \point in {v1,p,v2}
                \draw[fill=black] (\point) circle (2pt);
    
            % Label points
            \node[anchor=east] at (v1) {$v_1$};
            \node[anchor=south] at (p) {$p$};
            \node[anchor=west] at (v2) {$v_2$};
         \end{tikzpicture}
        \subcaption{}
    \end{minipage}
    \begin{minipage}[b]{.4\textwidth}
        \centering
        \tdplotsetmaincoords{110}{20} % Adjust viewing angle
    \begin{tikzpicture}[tdplot_main_coords]
    
      % Define points
      \coordinate (v1) at (0,0,0);
      \coordinate (v2) at (2,0,0);
      \coordinate (p') at (1,1,2);
      \coordinate (v1') at (0,0,2);
      \coordinate (v2') at (2,0,2);

      % Draw rectangles (background)
      \draw[fill=orange!30,opacity=0.6] (-1,-1,0) -- (3,-1,0) -- (3,2,0) -- (-1,2,0) -- cycle;

      % Draw triangles
      \draw[fill=blue!20,opacity=0.5] (v1) -- (v2) -- (p') -- cycle; % Triangle [v1,v2,p]
      \draw[fill=red!20,opacity=0.5] (v1) -- (v1') -- (p') -- cycle; % Triangle [v1,v'1,p]
      \draw[fill=green!20,opacity=0.5] (v2) -- (v2') -- (p') -- cycle; % Triangle [v2,v'2,p]

      % Draw rectangles (foreground)
      \draw[fill=orange!30,opacity=0.6] (-1,-1,2) -- (3,-1,2) -- (3,2,2) -- (-1,2,2) -- cycle;
      
      % Draw points
      \foreach \point in {v1,v2,p',v1',v2'}
        \draw[fill=black] (\point) circle (0.5mm);
    
      % Label points
      \node[anchor=east] at (v1) {$v_1$};
      \node[anchor=west] at (v2) {$v_2$};
      \node[anchor=south] at (p') {$p'$};
      \node[anchor=east] at (v1') {$v'_1$};
      \node[anchor=west] at (v2') {$v'_2$};

       \node[anchor=east] at (-1,-1,0) {$P_0$};
       \node[anchor=east] at (-1,-1,2) {$P_c$};
    \end{tikzpicture}
        \subcaption{}
        \label{fig:schematic_triangulation}
    \end{minipage}
    \caption{Deforming an edge and a corresponding triangulation as given in \cite{goertzen2024constructing}: (a) start with an initial edge with vertices $v_1,v_2$, (b) introducing intermediate point $p$ results in two edges with vertices $v_1,p$ and $v_2,p$, (c) interpolating between edges by setting $v_1'=v_1+(0,0,c)^\intercal,v_2'=v_2+(0,0,c)^\intercal,p'=p+(0,0,c)^\intercal$ for some $c>0$. Note that the points $p',v_1',v_2'$  and the points $v_1,v_2,p$ lie in the planes $P_c=\{(x,y,c)\mid x,y\in \R \}$ and $P_0\{(x,y,0)\mid x,y\in \R \}$, respectively.
    }
    \label{fig:deforming_paths}
\end{figure}
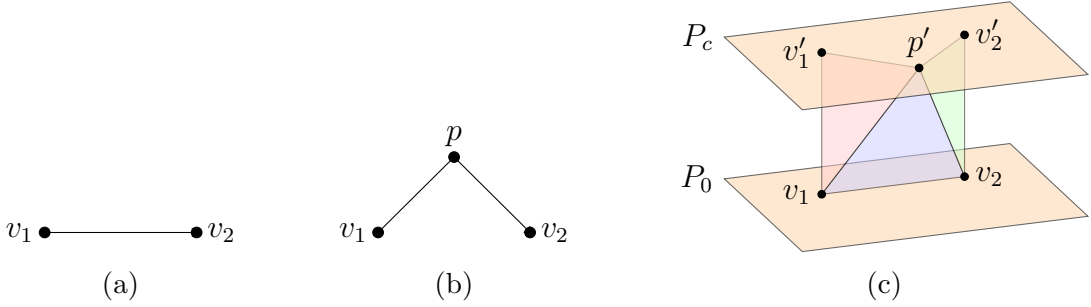

For each vertex $v$ of $T$, we obtain a quadruple $(e,v,n,g)$, which we call an \emph{edge representing quadruple}. In order to put together all contacts of blocks within the assembly, we define the following polytope:
\begin{align*}
\mathrm{IA}_{X,G}= & \{\left(\gamma_{g}\right)_{g\in G}\in\mathbb{R}^{6\times G}\mid\left(e,v,n,g\right)\text{ edge representing quadruple, }\lvert \mathrm{supp}(\left(\gamma_{g}\right)_{g\in G}) \rvert <\infty\\
 & \left(-((R(v)+t)\times R(n))^\intercal,-R(n)^\intercal,\left((R(v)+t)\times R(n)\right)^\intercal,R(n)^\intercal\right)\cdot\left(\gamma_{g'},\gamma_{g'\cdot g}\right)\geq 0,\\ & \forall g'=(R,t)\in G\}.
\end{align*}
Since each contact triangle comes from an edge, we have described all contact faces in the polytope above and thus have the following result.
\begin{proposition}\label{proposition:infinite_polytope}
    The polytope $\mathrm{IA}_{X,G}$ is the interlocking polytope as defined in Definition \ref{def:infinitesimal_interlocking}. 
\end{proposition}
\begin{proof}
    We need to show that the inequalities defining the polytope $\mathrm{IA}_{X,G}$ are exactly those coming from the interlocking matrix $A$ in Definition \ref{def:infinitesimal_interlocking}. Let $p$ be a contact point of a contact face with normal $n$ two blocks $X_{g'},X_{g''}$ coming from a deformed edge $e$. It follows that $g'^{-1}(e) \in \{e_1,\dots,e_n \}$ and we write $g=g'^{-1} \cdot g''$. By the definition of the underlying assembly, we have that $(R(e)+t,R(v)+t,R(n),g)$ is an edge representing quadruple, where $g'^{-1}=(R,t)\in G$. Thus, we obtain all inequalities of $A$ and vice versa, each inequality in the definition of $\mathrm{IA}_{X,G}$ comes from $A$.
\end{proof}

\subsection{The Kernel of the Infinitesimal Interlocking Matrix}

In this section, we show that no sliding motions are possible for the blocks constructed in \cite{goertzen2024constructing} as long as they fulfil Criterion \ref{criterion:interlocking}. For this, let $(X_g)_{g\in G}$ be an assembly constructed with the methods in \cite{goertzen2024constructing} and fulfilling Criterion \ref{criterion:interlocking}. Then we compute its infinitesimal interlocking matrix, which we call $A$.

\begin{lemma}\label{lemma:one_side}
Let $\left(\gamma_{g}\right)_{g\in G}\in\mathbb{R}^{6\times G}$ be a vector (with finite support) of infinitesimal motions for each block. A ``deformed side'' of a block means that for the bottom plane corresponding to the tiling with $F$, there is an edge with corresponding intermediate point deforming, the edge leading to three faces with distinct normal vectors having contact. Assume that one ``deformed side'' of a block $X$ has contact to another block $X'$, which is not moving. Then we can simplify the computation of the kernel to a reduced matrix. 
\end{lemma}

\begin{proof}

    Let $e\subset \R^2$ be an edge with vertices $v_1,v_2\in \R^2$, and let $p\in \R^2$ be a point. We embed the vertices $v_1,v_2$ into $\R^3$ by appending a $0$ and let $v_1',v_2'$ be the points which are obtained by appending $1$ instead of $0$ and also embed $p$ into $\R^3$ by appending a $1$. We then consider the three triangles in $\R^3$ given by the points $T_2:=\{v_1,v_1',p \},T_1:=\{v_1,v_2,p \},T_3:=\{v_2,v_2',p \}$ and normal vectors $n_i$ of $T_i$ given by $n_1:=(v_2-v_1)\times (p-v_1),n_2:=(v_1'-v_1)\times (p-v_1), n_3:=(v_2'-v_2)\times (p-v_2)$. We obtain the following $9\times6$ sub-matrix of $A$ of the form corresponding to the contacts of blocks $X$ and $X'$:
\[
A_e^\intercal = 
\begin{pmatrix}
p \times n_{1} & v_{1} \times n_{1} & v_{2} \times n_{1} & p \times n_{2} & v_{1} \times n_{2} & v_{1}' \times n_{2} & p \times n_{3} & v_{2} \times n_{3} & v_{2}' \times n_{3} \\
n_{1} & n_{1} & n_{1} & n_{2} & n_{2} & n_{2} & n_{3} & n_{3} & n_{3}
\end{pmatrix}.
\]
Since we need to show that $A_e$ has rank $6$, we can use row operation on $A_e$ to obtain the following matrix with the same rank
\[
B_{e}^\intercal = \begin{psmallmatrix}
p \times n_{1} & (v_{1}-p) \times n_{1} & (v_{2}-p) \times n_{1} & p \times n_{2} & (v_{1}-p) \times n_{2} & (v_{1}'-p) \times n_{2} & p \times n_{3} & (v_{2}-p) \times n_{3} & (v_{2}'-p) \times n_{3} \\
n_{1} & 0 & 0 & n_{2} & 0 & 0 & n_{3} & 0 & 0
\end{psmallmatrix}.
\]
Now, it suffices to show that $n_1,n_2,n_3$ are linearly independent and one of the following matrices have rank $3$:
\[
\begin{pmatrix}
\left((v_{1}-p)\times n_{1}\right)^\intercal\\
\left((v_{1}-p)\times n_{2}\right)^\intercal\\
\left((v_{2}-p)\times n_{3}\right)^\intercal
\end{pmatrix}\text{ or }\begin{pmatrix}
\left((v_{2}-p)\times n_{1}\right)^\intercal\\
\left((v_{1}-p)\times n_{2}\right)^\intercal\\
\left((v_{2}-p)\times n_{3}\right)^\intercal
\end{pmatrix}.
\]
For a $3\times3$ matrix $M$ with $\det M \neq 0$ it holds that
$$ (Ma)\times (Mb)= \mathrm{cof}\left(M\right)\cdot (a\times b),$$
where $\mathrm{cof}\left(M\right)=\det \left(M\right) \cdot (M^{-1})^\intercal$ is the \emph{cofactor matrix} of $M$. Hence, we can assume that $v_1=(0,0)^\intercal,v_2=(1,0)^\intercal$ and the matrices simplify to

\begin{align*}
M_{1}= & \begin{pmatrix}-p_{2}^{2}-1 & p_{1} & p_{1}\\
-\left(1-p_{1}\right)p_{2} & p_{2} & p_{2}\\
p_{1}-1 & -p_{1}^{2}-p_{2}^{2} & -\left(1-p_{1}\right)^{2}-p_{2}^{2}
\end{pmatrix},\\
M_{2}= & \begin{pmatrix}-p_{2}^{2}-1 & p_{1} & p_{1}-1\\
p_{1}p_{2} & p_{2} & p_{2}\\
p_{1} & -p_{1}^{2}-p_{2}^{2} & -\left(1-p_{1}\right)^{2}-p_{2}^{2}
\end{pmatrix}.
\end{align*}
Then we obtain the following determinants of the matrices dependent on the values of $p$:
$$\det M_1=-p_2 (p_1-1)(p_1^2+p_2^2+1) , \det M_2= -p_1 p_2 (p_1^2+p_2^2-2 p_1 +2).$$ Since $p$ is not contained on the line $v_1,v_2$, we can assume that $p_2\neq 0$ and if $p_1=1$ it follows that $M_2$ has non-zero determinant and if $p_1\neq 1$ it follows that $M_1$ has non-zero determinant.
\end{proof}

\begin{proposition}\label{proposition:kernel}
    The infinitesimal interlocking matrix $A$ has trivial kernel if and only if Criterion \ref{criterion:interlocking} holds.
\end{proposition}

\begin{proof}
    In the discussion preceding Criterion \ref{criterion:interlocking}, it is shown that we can find a non-trivial admissible motion shifting out blocks if there is a finite set $\empty \neq H \subset G$ with $H(F)=H(F')$. Moreover, this motion correspond to the translation $(0,0,1)$ and indeed gives rise to a sliding motion.

    Assuming the opposite, we show that there are no sliding motions, i.e.\ the matrix $A$ has trivial kernel.
    Let $x$ with finite support and $Ax=0$, we identify $x$ with the family of infinitesimal motions $(\gamma_g)_{g\in G}$ such that $\gamma_g=0$ for almost all $g\in G$. We show by induction that $x=0$. For this, we need to prove that there is always an edge in the sense of Lemma \ref{lemma:one_side}. But since there is always at least one block having contact to a block which is moved by $x$, such an edge exists. For each non-open edge $e$, there is at least one block with points ``above'' $e$. Since the number of blocks that are considered is finite, we find an open edge $e$, since at least one edge admits deformations.
\end{proof}

\subsection{Interlocking Chains}

From the infinite polytope formulation, we obtain a matrix $A$ such that vectors $x$ with finite support and $Ax \geq 0$ correspond to admissible infinitesimal motions. We show that, in the case where $G$ is of type p1 and only translations are considered, the inequality $(Ax)_i > 0$ for some row index $i$ implies that $x$ has infinite support. Therefore, it suffices to consider $x$ with $Ax = 0$. This can be demonstrated by proving the existence of an infinite intersection chain.

\begin{theorem}\label{thm:p1_translational}
        Any assembly $(X_g)_{g\in G}$ coming from the construction of the previous section with $G$ of type p1 is translational interlocked if and only if Criterion \ref{criterion:interlocking} is fulfilled.
\end{theorem}

\begin{proof}
    If there exists a set $H$ as in Criterion \ref{criterion:interlocking}, we can simply shift out the corresponding blocks using upward translations. Proposition \ref{proposition:kernel} implies that there are no non-trivial motions inside the kernel of the interlocking matrix. Thus, it suffices to rule out motions $\gamma$ such that there exists $k$ with $(A\cdot \gamma)_k>0$. If there exists such an entry, it follows that there exists a block $X_i$ with non-trivial motion $t_i$ and a normal $n$ of a face with $t_i \cdot n >0$. Then for any face of the block $X_i$ with normal $n'$ such that $t_i\cdot n'<0$ implies that the neighbouring block at this face has to move as-well. This automatically leads to an infinite chain.
\end{proof}

\begin{corollary}\label{corollary:finite_interlocking}
Let $\left(X_{i}\right)_{i\in I}$ be an infinite interlocking assembly and $J\subset I$ a finite subset in $I$, such that the elements in $J$ can be written as a tuple $\left(j_{1},\dots,j_{n}\right)$ such that the blocks $X_{j_{l}},X_{j_{l+1}}$ for $l=1,\dots,n-1$ and the blocks $X_{j_1},X_{j_n}$ share a common face. Then there is a finite subset $J\subset\tilde{I}\subset I$ such that $\left(\left(X_{i}\right)_{i\in\tilde{I}},J\right)$is a topological interlocking assembly. Moreover, this subset can be chosen to be the blocks inside the region spanned by the blocks corresponding to the elements in $J$.
\end{corollary} 

\subsection{Interlocking Assemblies with RhomBlocks}

In this section, we prove that any infinite tiling with lozenges leads to two interlocking assemblies with RhomBlocks. In \cite{goertzen2024constructing}, a triangulation of the surface of the RhomBlock is given with coordinates
\[
\begin{array}{cccccc}
v_{1} & v_{2} & v_{3} & v_{4} & v_{5} \\
\left(0,0,0\right)^\intercal & \left(\frac{1}{2},\frac{\sqrt{3}}{2},0\right)^\intercal & \left(1,0,0\right)^\intercal & \left(\frac{1}{2},-\frac{\sqrt{3}}{2},0\right)^\intercal & \left(0,0,\sqrt{\frac{2}{3}}\right)^\intercal \\
v_{6} & v_{7} & v_{8} & v_{9} & v_{10}\\
\left(\frac{1}{2},\frac{\sqrt{3}}{6},\sqrt{\frac{2}{3}}\right)^\intercal & \left(1,0,\sqrt{\frac{2}{3}}\right)^\intercal & \left(1,-\frac{\sqrt{3}}{3},\sqrt{\frac{2}{3}}\right)^\intercal & \left(\frac{1}{2},-\frac{\sqrt{3}}{2},\sqrt{\frac{2}{3}}\right)^\intercal & \left(0,-\frac{\sqrt{3}}{3},\sqrt{\frac{2}{3}}\right)^\intercal
\end{array}
\]
and corresponding vertices of faces:
\begin{align*}
    &[[ 1, 2, 3 ], [ 1, 3, 4 ], [ 1, 5, 6 ], [ 1, 2, 6 ], [ 2, 3, 6 ], [ 3, 6, 7 ], [ 3, 7, 8 ], [ 3, 4, 8 ], [ 4, 8, 9 ],  [ 4, 9, 10 ],\\&\phantom{[}  [ 1, 4, 10 ], [ 1, 5, 10 ], [ 5, 6, 7 ], [ 5, 9, 10 ], [ 7, 8, 9 ], [ 5, 7, 9 ]].\notag
\end{align*}

As shown in \cite{goertzen2024constructing}, we obtain for each tesselations with lozenges to assemblies with RhomBlocks. The simple geometry can be exploited to obtain the following result, where we use the fact that the vertical walls of the RhomBlock come in parallel pairs.

\begin{theorem}
    For a given tessellation of the plane with lozenges, we can associate two infinite interlocking RhomBlock assembly. 
\end{theorem}

\begin{proof}
    In \cite{goertzen2024constructing} it is shown that we can associate two RhomBlock assemblies to a given lozenges tiling. 
In order to prove that this assembly is indeed an interlocking assembly, we consider its infinitesimal interlocking matrix $A$ and show that for any $\gamma$ with finite support, and $A\cdot \gamma \geq 0 $ it follows that $\gamma=0$. This is proven in two steps: first we show that for any such $\gamma$ there is no $i$ with $(A\cdot \gamma)_i >0$ and then we show that the kernel of the matrix $A$ is trivial. So assume that, there exists $i$ with $(A\cdot \gamma)_i >0$. If $i$ is associated to a vertex of a vertical face, we can construct an infinite chain as follows:  the entry $(A\cdot \gamma)_i >0$ corresponds to an inequality of the form $(-(p\times n)^\intercal,-n^\intercal,(p\times n)^\intercal,n^\intercal)\cdot (\gamma',\gamma)>0$ according to Definition \ref{def:infinitesimal_interlocking}. If $((p\times n)^\intercal,n^\intercal)\cdot \gamma>0$ it follows that $((p\times -n)^\intercal,-n^\intercal)\cdot \gamma=(((p-n)\times -n)^\intercal,-n^\intercal)\cdot \gamma<0$ and thus there must be $\gamma''$ with $(((p-n)\times n)^\intercal,n^\intercal)\cdot \gamma''>0$. Proceeding iteratively, we obtain an infinite chain of moving blocks, which is not possible as we assumed finite support and thus $(A\cdot \gamma)_i =0$ for any point associated to a vertical face. This implies that $\gamma$ consists only of upward translations. If $i$ then belongs to a tilted face, it follows that upward translation lead to another infinite chain. Hence, the case $A\gamma=0$ remains. This follows from the more general Lemma in Section \ref{sec:proof}. In total, we have that the RhomBlock assembly is indeed an interlocking assembly.
\end{proof}

\section{Conclusion and Outlook}

In this work, we first establish the mathematical foundations of interlocking assemblies. Next, we prove that a method, first introduced in \cite{wang_computational_2021}, can be used to verify the interlocking property. This is showcased for assemblies constructed in \cite{goertzen2024constructing} with a special emphasis on assemblies with RhomBlocks that arise from tessellations with lozenges.

In the literature, several other construction methods for candidates for interlocking assemblies are known, such as those involving convex bodies in \cite{EstDysArcPasBelKanPogodaevConvex}. As a next step into the general theory of interlocking assemblies, it might be interesting to apply the developed theory in this work to verify that these assemblies indeed satisfy the interlocking definition given in Definition \ref{def:interlocking} by means of Proposition \ref{proposition:infinitesimal}.

\section{Acknowledgements}
This work was funded by the Deutsche Forschungsgemeinschaft (DFG, German Research Foundation) – SFB/TRR 280. Project-ID: 417002380. 

\newpage
\bibliographystyle{plain}
\bibliography{cite} % Specify the BibTeX file without the file extension

\begin{thebibliography}{10}

\bibitem{akleman_generalized_2020}
Ergun Akleman, Vinayak~R. Krishnamurthy, Chia-An Fu, Sai~Ganesh Subramanian, Matthew Ebert, Matthew Eng, Courtney Starrett, and Haard Panchal.
\newblock Generalized abeille tiles: {Topologically} interlocked space-filling shapes generated based on fabric symmetries.
\newblock {\em Computers \& Graphics}, 89:156--166, 2020.

\bibitem{GoertzenBridges}
Reymond Akpanya, Tom Goertzen, Sebastian Wiesenhuetter, Alice~C. Niemeyer, and J\"{o}rg Noennig.
\newblock {Topological Interlocking, Truchet Tiles and Self-Assemblies: A Construction-Kit for Civil Engineering Design}.
\newblock In Judy Holdener, Eve Torrence, Chamberlain Fong, and Katherine Seaton, editors, {\em Proceedings of Bridges 2023: Mathematics, Art, Music, Architecture, Culture}, pages 61--68, Phoenix, Arizona, 2023. Tessellations Publishing.

\bibitem{EstDysArcPasBelKanPogodaevConvex}
A.~Belov-Kanel, Arcady Dyskin, Yuri Estrin, E.~Pasternak, and I.~Ivanov-Pogodaev.
\newblock Interlocking of convex polyhedra: towards a geometric theory of fragmented solids.
\newblock {\em Moscow Mathematical Journal}, 10, 01 2009.

\bibitem{InterlockingPuzzles}
Rulin Chen, Ziqi Wang, Peng Song, and Bernd Bickel.
\newblock Computational design of high-level interlocking puzzles.
\newblock {\em ACM Trans. Graph.}, 41(4), jul 2022.

\bibitem{HarmonicAnalysisApplied}
G.S. Chirikjian and A.B. Kyatkin.
\newblock {\em Harmonic {Analysis} for {Engineers} and {Applied} {Scientists}: {Updated} and {Expanded} {Edition}}.
\newblock Dover {Books} on {Mathematics}. Dover Publications, 2016.

\bibitem{dyskin_new_2001}
A.~V. Dyskin, Y.~Estrin, A.~J. Kanel-Belov, and E.~Pasternak.
\newblock A new concept in design of materials and structures: assemblies of interlocked tetrahedron-shaped elements.
\newblock {\em Scripta Materialia}, 44(12):2689--2694, 2001.

\bibitem{a_v_dyskin_topological_2003}
A.~V. Dyskin, Y.~Estrin, A.~J. Kanel-Belov, and E.~Pasternak.
\newblock Topological interlocking of platonic solids: {A} way to new materials and structures.
\newblock {\em Philosophical Magazine Letters}, 83(3):197--203, 2003.

\bibitem{dyskin_principle_2005}
A.~V. Dyskin, Y.~Estrin, E.~Pasternak, H.~C. Khor, and A.~J. Kanel-Belov.
\newblock The principle of topological interlocking in extraterrestrial construction.
\newblock {\em Acta Astronautica}, 57(1):10--21, 2005.

\bibitem{EstrinDesignReview}
Yuri Estrin, Vinayak~R. Krishnamurthy, and Ergun Akleman.
\newblock Design of architectured materials based on topological and geometrical interlocking.
\newblock {\em Journal of Materials Research and Technology}, 15:1165--1178, 2021.

\bibitem{fallacara_topological_2019}
Giuseppe Fallacara, Maurizio Barberio, and Micaela Colella.
\newblock Topological {Interlocking} {Blocks} for {Architecture}: {From} {Flat} to {Curved} {Morphologies}.
\newblock In Yuri Estrin, Yves Bréchet, John Dunlop, and Peter Fratzl, editors, {\em Architectured {Materials} in {Nature} and {Engineering}: {Archimats}}, pages 423--445. Springer International Publishing, Cham, 2019.

\bibitem{glickman_g-block_1984}
Michael Glickman.
\newblock The {G}-block system of vertically interlocking paving.
\newblock In {\em Second {International} {Conference} on {Concrete} {Block} {Paving}}, pages 10--12, 1984.

\bibitem{goertzen2024constructing}
Tom {Goertzen}.
\newblock {Constructing Interlocking Assemblies with Crystallographic Symmetries}.
\newblock {\em arXiv e-prints}, May 2024.
\newblock arXiv:2405.15080.

\bibitem{GoertzenFIB}
Tom Goertzen, Alice Niemeyer, and Wilhelm Plesken.
\newblock Topological interlocking via symmetry.
\newblock In {\em Proc. of the 6th fib International Congress 2022}, page 1235 – 1244. Novus Press, Oslo, Norway, 2022.

\bibitem{GrunbaumPolytopes}
Branko Gr{\"u}nbaum.
\newblock {\em Convex polytopes}, volume 221 of {\em Graduate Texts in Mathematics}.
\newblock Springer-Verlag, New York, second edition, 2003.
\newblock Prepared and with a preface by Volker Kaibel, Victor Klee and G\"{u}nter M. Ziegler.

\bibitem{ITA2002}
IUCr.
\newblock {\em International Tables for Crystallography, Volume A: Space Group Symmetry}.
\newblock International Tables for Crystallography. Kluwer Academic Publishers, Dordrecht, Boston, London, 5. revised edition edition, 2002.

\bibitem{scipy}
Eric Jones, Travis Oliphant, Pearu Peterson, et~al.
\newblock {SciPy}: Open source scientific tools for {Python}, 2001--2024.

\bibitem{DifferentialGeometryKobayashi}
Shoshichi Kobayashi and Katsumi Nomizu.
\newblock {\em Foundations of differential geometry. {V}ol. {I}}.
\newblock Wiley Classics Library. John Wiley \& Sons, Inc., New York, 1996.
\newblock Reprint of the 1963 original, A Wiley-Interscience Publication.

\bibitem{lecci_design_2021}
Francesca Lecci, Cecilia Mazzoli, Cristiana Bartolomei, and Riccardo Gulli.
\newblock Design of {Flat} {Vaults} with {Topological} {Interlocking} {Solids}.
\newblock {\em Nexus Network Journal}, 23(3):607--627, September 2021.

\bibitem{piekarski_floor_2020}
Maciej Piekarski.
\newblock Floor {Slabs} {Made} from {Topologically} {Interlocking} {Prefabs} of {Small} {Size}.
\newblock {\em Buildings}, 10(4), 2020.

\bibitem{stuttgen_modular_2023}
Sascha Stüttgen, Reymond Akpanya, Birgit Beckmann, Rostislav Chudoba, Daniel Robertz, and Alice~C. Niemeyer.
\newblock Modular {Construction} of {Topological} {Interlocking} {Blocks}—{An} {Algebraic} {Approach} for {Resource}-{Efficient} {Carbon}-{Reinforced} {Concrete} {Structures}.
\newblock {\em Buildings}, 13(10), 2023.

\bibitem{wang_computational_2021}
Ziqi Wang.
\newblock {\em Computational {Analysis} and {Design} of {Structurally} {Stable} {Assemblies} with {Rigid} {Parts}}.
\newblock PhD thesis, EPFL, 2021.

\bibitem{wang_design_2019}
Ziqi Wang, Peng Song, Florin Isvoranu, and Mark Pauly.
\newblock Design and {Structural} {Optimization} of {Topological} {Interlocking} {Assemblies}.
\newblock {\em ACM Trans. Graph.}, 38(6), November 2019.
\newblock Place: New York, NY, USA Publisher: Association for Computing Machinery.

\bibitem{wilson_geometric_1992}
Randall~H. Wilson.
\newblock {\em On {Geometric} {Assembly} {Planning}}.
\newblock {PhD} {Thesis}, Stanford University, Stanford, CA, USA, 1992.

\end{thebibliography}

\end{document}